\documentclass[11pt]{article}
\usepackage{geometry}
\usepackage{xspace,url}
\usepackage{longtable}
\usepackage{amsmath, amsthm, amssymb}
\usepackage{setspace}
\usepackage{fancyhdr}
\usepackage{colortbl}
\usepackage{extramarks}
\usepackage{booktabs}
\usepackage{graphicx}
\usepackage{ifthen}
\usepackage{listings}
\usepackage{courier}
\usepackage{array}
\usepackage{tabularx}
\usepackage{pdfpages}
\usepackage{verbatim}
\usepackage{amsthm}
\usepackage{titling}
\usepackage{amssymb}
\usepackage{subfigure}
\usepackage{algorithm}
\usepackage{algorithmic}
\usepackage{ulem}

\usepackage{authblk}

\newtheorem{theorem}{Theorem}
\newtheorem{lemma}[theorem]{Lemma}
\newtheorem{corollary}[theorem]{Corollary}
\newtheorem{definition}[theorem]{Definition}

\newcommand{\beq}{\begin{equation}}
\newcommand{\eeq}{\end{equation}}
\newcommand{\R}{\mathbb{R}}
\newcommand{\E}{\mathbb{E}}
\newcommand{\Prio}{\mbox{\rm Pri}\emptyset}
\newcommand{\Dualo}{\mbox{\rm Dual}\emptyset}
\newcommand{\dist}{\mbox{\rm dist}}

\newcommand{\exclude}[1]{}
\newcommand{\set}[1]{\{#1\}}
\newcommand{\setof}[2]{\{#1\mid#2\}}
\def\eqnok#1{(\ref{#1})}

\title{An Approximate, Efficient LP Solver\\ for LP Rounding\thanks{This
manuscript is a full version of \cite{Sridhar2013NIPS}.}}
\author[1]{Srikrishna Sridhar}
\author[1]{Victor Bittorf}
\author[1]{Ji Liu}
\author[1]{Ce Zhang}
\author[2]{Christopher R\'e}
\author[1]{Stephen J. Wright}
\affil[1]{Computer Sciences, University of Wisconsin-Madison}
\affil[2]{Computer Science, Stanford University}

\newcommand{\cC}{\mathcal{C}}

\newcommand{\cS}{\mathcal{S}}

\newcommand{\ALG}{\mbox{ALG}}

\newcommand{\bB}{\{0,1\}}

\newcommand{\bRp}{\mathbb{R}^{+}}

\newcommand{\bfone}{\mbox{\bf 1}}

\newcommand{\norm}[1]{\lVert #1 \lVert}

\newcommand{\cplex}{Cplex\xspace}

\newcounter{count}
\newcommand{\spc}[1]{ \ifthenelse{\equal{#1}{0}}{}{ \setcounter{count}{#1} \addtocounter{count}{-1} \  \spc{\arabic{count}} }}

\newcommand{\hc}[1]{#1}
\providecommand{\e}[1]{\ensuremath{\times 10^{#1}}}

\newcommand{\Lmax}{L_{\max}}

\begin{document}

\maketitle

\begin{abstract}
Many problems in machine learning can be solved by rounding the solution of an
appropriate linear program (LP). This paper shows that we can recover
solutions of comparable quality by rounding an approximate LP solution instead
of the exact one. These approximate LP solutions can be computed efficiently by
applying a parallel stochastic-coordinate-descent method to a quadratic-penalty
formulation of the LP. We derive worst-case runtime and solution quality
guarantees of this scheme using novel perturbation and convergence analysis.
Our experiments demonstrate that on such combinatorial problems as vertex
cover, independent set and multiway-cut, our approximate rounding scheme is up
to an order of magnitude faster than \cplex (a commercial LP solver) while
producing solutions of similar quality.

\end{abstract}

\section{Introduction}
A host of machine-learning problems can be solved effectively as
approximations of such NP-hard combinatorial problems as set cover,
set packing, and
multiway-cuts~\cite{Koval:1976:lr,ravikumar:2010:jmlr,Vazirani,
  Lempitsky:2007:cvpr}. A popular scheme for solving such problems is
called LP rounding~\cite[chs.~12-26]{Vazirani}, which consists of the
following three-step process: (1) construct an integer (binary) linear
program (IP) formulation of a given problem; (2) relax the IP to an LP
by replacing the constraints $x \in \{0,1\}$ by $x \in[0,1]$; and (3)
round an optimal solution of the LP to create a feasible solution for
the original IP problem. LP rounding is known to work well on a range
of hard problems, and comes with theoretical guarantees for runtime
and solution quality.
 
The Achilles' heel of LP-rounding \hc{is that it requires solutions} of LPs of
possibly extreme scale. Despite decades of work on \hc{LP} solvers, including
impressive advances during the 1990s, commercial codes such as \cplex or Gurobi
may not be capable of handling problems of the required scale. In this work, we
propose an approximate LP solver suitable for use in the LP-rounding approach,
for very large problems. \hc{Our intuition is that in LP rounding, since we
ultimately round the LP to obtain an approximate solution of the combinatorial
problem, a crude solution of the LP may suffice.  Hence, an approach that can
find approximate solutions of large LPs quickly may be suitable, even if it is
inefficient for obtaining highly accurate solutions.}  

\hc{This paper focuses} on the theoretical and algorithmic aspects of finding
approximate solutions to an LP, for use in LP-rounding schemes.  Our three main
technical contributions are as follows: First, we show that one can
approximately solve large \hc{LPs} by forming \hc{convex} quadratic programming
(QP) approximations, then applying stochastic coordinate descent to these
approximations. Second, we derive a novel convergence analysis of our method,
based on Renegar's perturbation theory for linear programming~\cite{Ren94b}.
Finally, we derive bounds \hc{on runtime as well} as worst-case approximation
ratio of our rounding schemes.  \hc{Our experiments demonstrate that our
approach, called Thetis, produces solutions of comparable quality to
state-of-the-art approaches on such tasks as noun-phrase chunking and entity
resolution. We also demonstrate, on three different classes of combinatorial
problems, that Thetis can outperform \cplex (a state-of-the-art commercial LP
and IP solver) by up to an order of magnitude in runtime, while achieving
comparable solution quality}.

\paragraph*{Related Work.}
\hc{Recently, there has been some focus on the connection between LP
relaxations and maximum a posteriori (MAP) estimation
problems~\cite{Sanghavi2007a,ravikumar:2010:jmlr}}. Ravikumar et.\ al
~\cite{ravikumar:2010:jmlr} \hc{proposed rounding schemes for iterative LP
solvers to facilitate MAP inference in graphical models}.  In contrast, we
propose to use stochastic descent methods to solve a QP relaxation; this allows
us to take advantage of recent results on asynchronous parallel methods of this
type~\cite{Niu2011,LiuW13b}.  Recently, Makari et.\ al~\cite{Makari:2013:VLDB}
propose an intriguing parallel scheme for packing and covering problems. \hc{In
contrast,} our results apply to more general LP relaxations, including
set-partitioning problems like multiway-cut. \hc{Additionally, the runtime of
our algorithm} is less sensitive to approximation error. For an error
$\varepsilon$, the bound on runtime of the algorithm in \cite{Makari:2013:VLDB}
grows as $\varepsilon^{-5}$, while the bound on our algorithm's runtime grows
as $\varepsilon^{-2}$.


\section{Background: Approximating NP-hard problems with LP Rounding}
\label{sec:vc:rounding} 

In this section, we review the theory of LP-rounding based
approximation schemes for NP-hard combinatorial problems.  We use the
vertex cover problem as an example, as it is the simplest nontrivial
setting that exposes the main ideas of this approach.

\paragraph*{Preliminaries.}
For a minimization problem $\Phi$, an algorithm $\ALG$ is an {\it
  $\alpha$-factor approximation for $\Phi$}, for some $\alpha > 1$, if
any solution produced by $\ALG$ has an objective value at most
$\alpha$ times the value of an optimal (lowest cost) solution.
For some problems, such as vertex cover, there is a constant-factor
approximation scheme ($\alpha=2$). For others, such as set
cover, the value of $\alpha$ can be as large as $O(\log N)$, where $N$
is the number of sets.

An LP-rounding based approximation scheme for the problem $\Phi$ first
{\it constructs} an IP formulation of $\Phi$ which we denote as
``$P$''.  This step is typically easy to perform, but the IP
formulation $P$ is, in theory, as hard to solve as the original
problem $\Phi$.  In this work, we consider applications in which the
only integer variables in the IP formulation are binary variables $x
\in \{0,1\}$.  The second step in LP rounding is a {\it relax / solve}
step: We relax the constraints in $P$ to obtain a linear program
$LP(P)$, replacing the binary variables with continuous variables in
$[0,1]$, then solve $LP(P)$.  The third step is to {\it round} the
solution of $LP(P)$ to an integer solution which is feasible for $P$,
thus yielding a candidate solution to the original problem $\Phi$. The
focus of this paper is on the relax / solve step, which is usually the
computational bottleneck in an LP-rounding based approximation scheme.

\paragraph*{Example: An Oblivious-Rounding Scheme For Vertex Cover.} 

\hc{Let $G(V,E)$ denote a graph with vertex set $V$ and undirected edges $E
\subseteq (V \times V)$.}  Let $c_v$ denote a nonnegative cost associated with
each vertex $v \in V$.  A vertex cover of a graph is a subset of $V$ such that
each edge $e \in E$ is incident to at least one vertex in this set. The {\it
minimum-cost} vertex cover is the one that minimizes the sum of terms $c_v$,
summed over the vertices $v$ belonging to the cover. Let us review the
``construct,'' ``relax / solve,'' and ``round'' phases of an LP-rounding based
approximation scheme applied to vertex cover.

In the ``construct'' phase, we introduce binary variables $x_v \in
\bB$, $\forall v \in V$, where $x_v$ is set to $1$ if the vertex $v
\in V$ is selected in the vertex cover and $0$ otherwise. The \hc{IP}
formulation is as follows:
\begin{equation}
\min_x \sum_{v \in V} c_v x_v \text{ s.t. } \; x_{u} + x_{v}  \geq 1 \text{ for } (u,v)  \in E \; \text{ and } \; x_v \in \{0,1\} \text{ for } v \in V. \label{model:vertex_cover}
\end{equation}
Relaxation yields the following LP
\begin{equation}
\min_x \sum_{v \in V} c_v x_v \text{ s.t. } \; x_{u} + x_{v}  \geq
1 \text{ for } (u,v)  \in E \; \text{ and } \; x_v \in [0,1] \text{ for } v \in
V.  \label{model:lp_vertex_cover} 
\end{equation}
A feasible solution of the LP relaxation \eqref{model:lp_vertex_cover} is
called a ``fractional solution'' of the original problem. In the ``round''
phase, we generate a valid vertex cover by simply choosing the vertices $v \in
V$ whose fractional solution $x_v \geq \frac{1}{2}$. It is easy to see that the
vertex cover generated by such a rounding scheme costs no more than twice the
cost of the fractional solution. If the fractional solution chosen for rounding
is an optimal solution of \eqref{model:lp_vertex_cover}, then we arrive at a
$2$-factor approximation scheme for vertex cover. We note here an important
property: The rounding algorithm can generate feasible integral solutions while
being {\it oblivious} of whether the fractional solution is an optimal solution
of \eqref{model:lp_vertex_cover}. We formally define the notion of an oblivious
rounding scheme as follows.


\begin{definition} \label{def:oblivious}
For a minimization problem $\Phi$ with an IP formulation $P$ whose LP
relaxation is denoted by $\mbox{LP}(P)$, a $\gamma$-factor `oblivious'
rounding scheme converts any feasible point $x_f \in \mbox{LP}(P)$ to
an integral solution $x_I \in P$ with cost at most $\gamma$ times the
cost of $\mbox{LP}(P)$ at $x_f$.
\end{definition}

Given a $\gamma$-factor {\it oblivious algorithm} $\ALG$ to the problem $\Phi$,
one can construct a $\gamma$-factor approximation algorithm for $\Phi$ by
\hc{using $\ALG$ to round an {\it optimal} fractional solution of
$\mbox{LP}(P)$}. When we have an approximate solution for $\mbox{LP}(P)$ that
is feasible for this problem, rounding can produce an $\alpha$-factor
approximation algorithm for $\Phi$ for a factor $\alpha$ slightly larger than
$\gamma$, where the difference between $\alpha$ and $\gamma$ takes account of
the inexactness in the approximate solution of $\mbox{LP}(P)$.  Many
LP-rounding schemes (\hc{including the scheme for vertex cover discussed in
Section~\ref{sec:vc:rounding}) are oblivious.}  We implemented the oblivious
LP-rounding algorithms in Figure~\ref{fig:schemes} and report experimental
results in Section~\ref{sec:experiments}.



\begin{figure}[t]
\footnotesize
\centering
\begin{tabular}{|c|l|l|}
 \hline
\multicolumn{1}{|c}{\bf Problem Family} & 
\multicolumn{1}{|c}{\textbf{Approximation Factor}} & 
\multicolumn{1}{|c|}{\bf Machine Learning Applications}\\
\hline
\hline
Set Covering & $\log(N)$~\cite{Srinivasan1999} & 
Classification~\cite{Bien2009}, Multi-object tracking~\cite{Wu2012}.\\ \hline 
Set Packing   & $e s +o(s)$~\cite{Bansal2012} &
MAP-inference~\cite{Sanghavi2007a}, Natural language~\cite{Kschischang2001}.
\\\hline 
Multiway-cut   & $3/2 - 1/k$ \cite{Rabini1998} & Computer
vision~\cite{Yuri2004}, Entity resolution~\cite{Lee2011}.
\\\hline 
Graphical Models & Heuristic &  Semantic role labeling~\cite{Roth:2005:icml},
Clustering~\cite{VanGael:2007:ijcai}. 
\\
\hline
\end{tabular}
\caption{LP-rounding schemes considered in this paper. The parameter $N$ refers
to the number of sets; $s$ refers to $s$-column sparse matrices; and $k$ refers
to the number of terminals. $e$ is the Euler's constant.}
\label{fig:schemes}
\end{figure}


\section{Main results} \label{sec:approxLP}
In this section, we describe how we can solve LP relaxations
approximately, in less time than traditional LP solvers, while still
preserving the formal guarantees of rounding schemes. We first define
a notion of approximate LP solution and discuss its consequences for
oblivious rounding schemes. We show that one can use a regularized
quadratic penalty formulation to compute these approximate LP
solutions. We then describe a stochastic-coordinate-descent (SCD)
algorithm for obtaining approximate solutions of this QP, and mention
enhancements of this approach, specifically, asynchronous parallel
implementation and the use of an augmented Lagrangian framework. Our
analysis yields a worst-case complexity bound for solution quality and
runtime of the entire LP-rounding scheme.


\subsection{Approximating LP Solutions}
Consider the LP in the following standard form
\begin{equation}
\min \, c^Tx \;\; \mbox{\rm s.t.} \; Ax=b, \;\; x \geq 0,
\label{eq:LP}
\end{equation}
where $c \in \R^n$, $b \in \R^m$, and $A \in \R^{m \times n}$ and its corresponding dual
\beq \label{eq:LPdual}
\max \, b^Tu \;\; \mbox{\rm s.t.}  \;\; c-A^Tu \ge 0.
\eeq
\hc{Let $x^{*}$ denote an optimal primal solution of \eqref{eq:LP}.  An
approximate LP solution $\hat{x}$ that we use for LP-rounding may be infeasible
and have objective value different from the optimum $c^T x^*$. We quantify the
inexactness in an approximate LP solution as follows.}
\begin{definition} \label{def:epsdel}
A point $\hat{x}$ is an $(\epsilon, \delta)$-approximate solution of the LP
\eqnok{eq:LP} if $\hat{x} \geq 0$ and there exists constants $\epsilon > 0$ and
$\delta > 0$ such that \[ \|A\hat{x} - b \|_{\infty} \leq \epsilon \quad \text{
and } \quad |c^T\hat{x} - c^Tx^*| \leq \delta | c^Tx^* |. \]  
\end{definition}
Using Definitions \ref{def:oblivious} and \ref{def:epsdel}, it is easy
to see that a $\gamma$-factor oblivious rounding scheme can round a
$(0, \delta)$ approximate solution to produce a feasible integral
solution whose cost is no more than $\gamma(1+\delta)$ times the
optimal solution of the $P$. The factor $(1+\delta)$ arises because
the rounding algorithm does not have access to an optimal fractional
solution. To cope with the infeasibility, we convert an
$(\epsilon,\delta)$-approximate solution to a $(0, \hat{\delta})$
approximate solution where $\hat{\delta}$ is not too large.
For vertex cover (\ref{model:lp_vertex_cover}), we prove the following
result in Appendix~\ref{app:roundings}. (Here, $\Pi_{[0,1]^n}(\cdot)$
denotes projection onto the unit hypercube in $\R^n$.)
\begin{lemma}
Let $\hat{x}$ be an $(\varepsilon, \delta)$ approximate solution to
the linear program~(\ref{model:lp_vertex_cover}) with $\varepsilon \in
[0,1)$. Then, $\tilde{x} = \Pi_{[0,1]^{n}}( (1-\varepsilon)^{-1}
  \hat{x} )$ is a $(0, \delta(1-\varepsilon)^{-1})$-approximate
  solution.\label{lem:infeasible:rounded:vc}
\end{lemma} 
\hc{Since $\tilde{x}$ is a feasible solution for
  \eqref{model:lp_vertex_cover}, the oblivious rounding scheme in
  Section~\ref{sec:vc:rounding} results in an
  $2(1+\delta(1-\varepsilon)^{-1})$ factor approximation algorithm. In
  general, constructing $(0, \hat{\delta})$ from $(\epsilon, \delta)$
  approximate solutions requires reasoning about the structure of a
  particular LP. In Appendix~\ref{app:roundings}, we establish
  statements analogous to Lemma~\ref{lem:infeasible:rounded:vc} for
  packing, covering and multiway-cut problems.}

\subsection{Quadratic Programming Approximation to the LP} \label{sec:qp_approx}

We consider the following regularized quadratic penalty approximation
to the LP \eqref{eq:LP}, parameterized by a positive constant $\beta$,
whose solution is denoted by $x(\beta)$:
\beq \label{eq:QP}
x(\beta) := \arg 
\min_{x \ge 0} \, f_{\beta} (x) := c^Tx - \bar{u}^T(Ax-b) + \frac{\beta}{2} \|Ax-b \|^2 + \frac{1}{2 \beta} \| x -\bar{x} \|^2,
\eeq
where $\bar{u} \in \R^m$ and $\bar{x} \in \R^n$ are arbitrary
vectors. (In practice, $\bar{u}$ and $\bar{x}$ may be chosen as
approximations to the dual and primal solutions of \eqnok{eq:LP}, or
simply set to zero.)
The quality of the approximation \eqnok{eq:QP} depends on the {\it
  conditioning} of underlying linear program \eqnok{eq:LP}, a concept
that was studied by Renegar~\cite{Ren94b}.  Denoting the data for
problem \eqnok{eq:LP} by $d:=(A,b,c)$, we consider perturbations
$\Delta d := (\Delta A, \Delta b, \Delta c)$ such that the linear
program defined by $d + \Delta d$ is primal infeasible. The primal
condition number $\delta_P$ is the infimum of the ratios $\| \Delta
d\| / \|d\|$ over all such vectors $\Delta d$. The dual condition
number $\delta_D$ is defined analogously. (Clearly both $\delta_P$ and
$\delta_D$ are in the range $[0,1]$; smaller values indicate poorer
conditioning.)  We have the following result, which is proven in the
supplementary material.
\begin{theorem} \label{th:pert1}
Suppose that $\delta_P$ and $\delta_D$ are both positive, and let $(x^*,u^*)$
be any primal-dual solution pair for \eqnok{eq:LP}, \eqnok{eq:LPdual}. If we
define $C_* := \max(\|x^*-\bar{x}\|,\|u^*-\bar{u}\|)$, then the unique solution
$x(\beta)$ of \eqnok{eq:QP} satisfies
\[
\| A x(\beta) - b \| \le (1/\beta) (1+\sqrt{2}) C_*, \quad
\| x(\beta) - x^* \| \le \sqrt{6} C_*.
\]
If in addition the parameter 
$$\beta  \ge \frac{10 C_*}{\| d\| \min(\delta_P,\delta_D)},$$ 
then we have
\[
|c^Tx^* - c^T x(\beta) | \le \frac{1}{\beta} \left[ \frac{25 C_*}{2
    \delta_P \delta_D} + 6 C_*^2 + \sqrt{6} \| \bar{x} \| C_* \right].
\]
\end{theorem}
In practice, we solve \eqnok{eq:QP} approximately, using an algorithm whose
complexity depends on the threshold $\bar{\epsilon}$ for which the objective is
accurate to within $\bar{\epsilon}$. That is, we seek $\hat{x}$ such that
\[
\beta^{-1} \| \hat{x} - x(\beta) \|^2 \le
f_{\beta}(\hat{x}) - f_{\beta}(x(\beta)) \le \bar{\epsilon},
\]
where the left-hand inequality follows from the fact that $f_{\beta}$ is
strongly convex with modulus $\beta^{-1}$.  If  we define
\beq \label{eq:C20}
\bar{\epsilon} := \frac{C_{20}^2}{\beta^3}, \quad 
C_{20} := \frac{25 C_*}{2 \|d\| \delta_P \delta_D},
\eeq
then by combining some elementary inequalities with the results of
Theorem~\ref{th:pert1}, we obtain the bounds
\[
| c^T \hat{x} - c^T x^* |  \le 
\frac{1}{\beta} \left[\frac{25 C_*}{\delta_P \delta_D} + 6C_*^2 + \sqrt{6} \| \bar{x} \| C_* \right], \quad
\| A \hat{x} - b \|  \le
\frac{1}{\beta} \left[ (1+\sqrt{2}) C_* + \frac{25 C_*}{2\delta_P \delta_D} \right].
\]
The following result is almost an immediate consequence.
\begin{theorem} \label{th:xh2}
Suppose that $\delta_P$ and $\delta_D$ are both positive and let $(x^*,u^*)$ be
any primal-dual optimal pair. Suppose that $C_*$ is defined as in
Theorem~\ref{th:pert1}.
Then for any given positive pair $(\epsilon,\delta)$, we have that
$\hat{x}$ satisfies the inequalities in Definition~\ref{def:epsdel} provided
that $\beta$ satisfies the following three lower bounds:
\begin{align*}
\beta & \ge  \frac{10 C_*}{\|d\| \min(\delta_P, \delta_D)}, \\
\beta &\ge \frac{1}{\delta |c^Tx^*|} 
\left[\frac{25 C_*}{\delta_P \delta_D} + 6C_*^2 + \sqrt{6} \| \bar{x} \| C_* \right], \\
\beta & \ge \frac{1}{\epsilon} \left[ (1+\sqrt{2}) C_* + \frac{25 C_*}{2\delta_P \delta_D} \right].
\end{align*}
\end{theorem}

For an instance of vertex cover with $n$ nodes and $m$ edges, we can
show that $\delta_{P}^{-1} = O(n^{1/2} (m + n)^{1/2})$ and
$\delta_{D}^{-1} = O((m + n)^{1/2})$ (see
Appendix~\ref{app:lp:conditioning}). The values $\bar{x} = \bfone$ and
$\bar{u}=\vec{0}$ yield $C_* \leq \sqrt{m}$. We therefore obtain
$\beta =O( m^{1/2} n^{1/2} (m+n) (\min \{\epsilon, \delta
|c^{T}x^*|\})^{-1})$.  


\subsection{Solving the QP Approximation: Coordinate Descent} \label{sec:SCD}

We propose the use of a stochastic coordinate descent (SCD) algorithm
\cite{LiuW13b} to solve \eqnok{eq:QP}. Each step of SCD chooses a component $i
\in \{1,2,\dotsc,n\}$ and takes a step in the $i$th component of $x$ along the
partial gradient of \eqnok{eq:QP} with respect to this component, projecting if
necessary to retain nonnegativity. This simple procedure depends on the
following constant $\Lmax$, which bounds the diagonals of the Hessian in the
objective of \eqnok{eq:QP}:
\beq \label{eq:Lmax}
\Lmax = \beta (\max_{i=1,2,\dotsc,n} A_{:i}^T A_{:i}) + \beta^{-1},
\eeq
where $A_{:i}$ denotes the $i$th column of $A$. Algorithm~\ref{alg:SCD}
describes the SCD method.
\begin{algorithm}[t] 
\caption{SCD  method for \eqnok{eq:QP}}
\label{alg:SCD}
\begin{algorithmic}[1]
\STATE Choose $x_0 \in \R^n$; $j \leftarrow 0$ \\
\LOOP
\STATE Choose $i(j) \in \{1,2,\dotsc,n\}$ randomly with equal probability; \\
\STATE Define $x_{j+1}$ from $x_j$ by setting
$ [x_{j+1}]_{i(j)} \leftarrow \max(0, [x_j]_{i(j)} -(1/\Lmax) [\nabla f_{\beta}(x_j)]_{i(j)})$, leaving other components unchanged; \\
\STATE $j \leftarrow j + 1$; \\
\ENDLOOP
\end{algorithmic}
\end{algorithm}
Convergence results for Algorithm~\ref{alg:SCD} can be obtained from
\cite{LiuW13b}. In this result, $\E(\cdot)$ denotes expectation over all the
random variables $i(j)$ indicating the update indices chosen at each
iteration. We need the following quantities:
\beq \label{eq:LL}
l := \frac{1}{\beta}, \quad
R := \sup_{j=1,2,\dotsc n} \|x_j - x(\beta) \|_2,
\eeq
where $x_j$ denotes the $j$th iterate of the SCD algorithm.  (Note
that $R$ bounds the maximum distance that the iterates travel from the
solution $x(\beta)$ of \eqnok{eq:QP}.)

\begin{theorem} \label{th:scd}
For Algorithm~\ref{alg:SCD} we have
\[
\E \|x_j - x(\beta) \|^2 + \frac{2}{\Lmax} \E (f_{\beta}(x_j) - f_{\beta}^*)
\le
\left( 1- \frac{l}{n(l+\Lmax)} \right)^j \left( R^2 +
\frac{2}{\Lmax} (f_{\beta}(x_0) - f_{\beta}^*) \right),
\]
where $f_{\beta}^* := f_{\beta} (x(\beta))$.  We obtain
high-probability convergence of $f_{\beta}(x_j)$ to $f_{\beta}^*$ in
the following sense: For any $\eta \in (0,1)$ and any small
$\bar{\epsilon}$, we have
$$
P (f_{\beta}(x_j)-f_{\beta}^* < \bar{\epsilon}) \ge 1-\eta,
$$
\ provided that 
\[
j \ge \frac{n(l+\Lmax)}{l} \left | \log
\frac{\Lmax}{2 \eta \bar{\epsilon}} 
\left( R^2 +
\frac{2}{\Lmax} (f_{\beta}(x_0) - f_{\beta}^*) \right)
\right|.
\]
\end{theorem}

\paragraph*{Worst-Case Complexity Bounds.} 
We now combine the analysis in Sections~\ref{sec:qp_approx} and \ref{sec:SCD}
to derive a worst-case complexity bound for our approximate LP solver.
Supposing that the columns of $A$ have norm $O(1)$, we have from
\eqnok{eq:Lmax} and \eqnok{eq:LL} that $l = \beta^{-1}$ and $\Lmax = O(\beta)$.
Theorem~\ref{th:scd} indicates that we require $O(n \beta^2)$ iterations to
solve~\eqref{eq:QP} (modulo a log term). For the values of $\beta$ described in
Section~\ref{sec:qp_approx}, this translates to a complexity estimate of $O(m^3
n^2 / \epsilon^2)$. 

In order to obtain the desired accuracy in terms of feasibility and
function value of the LP (captured by $\epsilon$) we need to solve the
QP to within the different, tighter tolerance $\bar{\epsilon}$
introduced in \eqnok{eq:C20}. Both tolerances are related to the
choice of penalty parameter $\beta$ in the QP. Ignoring here the
dependence on dimensions $m$ and $n$, we note the relationships $\beta
\sim \epsilon^{-1}$ (from Theorem \ref{th:xh2}) and $\bar{\epsilon}
\sim \beta^{-3} \sim \epsilon^3$ (from \eqnok{eq:C20}). Expressing all
quantities in terms of $\epsilon$, and using Theorem~\ref{th:scd}, we
see an iteration complexity of $\epsilon^{-2}$ for SCD (ignoring log
terms). The linear convergence rate of SCD is instrumental to this
favorable value. By contrast, standard variants of stochastic-gradient
descent (SGD) applied to the QP yield poorer complexity.  For
diminishing-step or constant-step variants of SGD, we see complexity
of $\epsilon^{-7}$, while for robust SGD, we see $\epsilon^{-10}$.
(Besides the inverse dependence on $\bar{\epsilon}$ or its square in
the analysis of these methods, there is a contribution of order
$\epsilon^{-2}$ from the conditioning of the QP.)


\subsection{Enhancements}

We mention two important enhancements that improve the efficiency of
the approach outlined above.  The first is an asynchronous parallel
implementation of Algorithm~\ref{alg:SCD} and the second is the use of
an augmented Lagrangian framework rather than ``one-shot''
approximation by the QP in \eqnok{eq:QP}.

\paragraph*{Asynchronous Parallel SCD.}

An asynchronous parallel version of Algorithm~\ref{alg:SCD}, described
in ~\cite{LiuW13b}, is suitable for execution on multicore,
shared-memory architectures. Each core, executing a single thread, has
access to the complete vector $x$. Each thread essentially runs its
own version of Algorithm~\ref{alg:SCD} independently of the others,
choosing and updating one component $i(j)$ of $x$ on each
iteration. Between the time a thread reads $x$ and performs its
update, $x$ usually will have been updated by several other threads.
Provided that the number of threads is not too large (according to
criteria that depends on $n$ and on the diagonal dominance properties of the
Hessian matrix), and the step size is chosen appropriately, the convergence
rate is similar to the serial case, and near-linear speedup is observed.



\paragraph*{Augmented Lagrangian Framework.} 

It is well known (see for example~\cite{Ber99,NocW06}) that the
quadratic-penalty approach can be extended to an augmented Lagrangian
framework, in which a sequence of problems of the form \eqnok{eq:QP}
are solved, with the primal and dual solution estimates $\bar{x}$ and
$\bar{u}$ (and possibly the penalty parameter $\beta$) updated between
iterations. Such a ``proximal method of multipliers'' for LP was
described in \cite{SJW14}.  We omit a discussion of the convergence
properties of the algorithm here, but note that the quality of
solution depends on the values of $\bar{x}$, $\bar{u}$ and $\beta$ at
the last iteration before convergence is declared. By applying
Theorem~\ref{th:xh2}, we note that the constant $C_*$ is smaller when
$\bar{x}$ and $\bar{u}$ are close to the primal and dual solution
sets, thus improving the approximation and reducing the need to
increase $\beta$ to a larger value to obtain an approximate solution
of acceptable accuracy.

\section{Experiments} \label{sec:experiments}

Our experiments address two main questions: (1) Is our approximate LP-rounding
scheme useful in graph analysis tasks that arise in machine learning?  and (2)
How does our approach compare to a state-of-the-art commercial solver?  We give
favorable answers to both questions.

\subsection{Is Our Approximate LP-Rounding Scheme Useful in Graph Analysis
Tasks?}

LP formulations have been used to solve MAP inference problems on graphical
models~\cite{ravikumar:2010:jmlr}, but general-purpose LP solvers have
rarely been used, for reasons of scalability. We demonstrate that the
rounded solutions obtained using Thetis are of comparable quality to those
obtained with state-of-the-art systems. We perform experiments on two different
tasks: entity linking and text chunking.  For each task, we produce a factor
graph \cite{Kschischang2001}, which consists of a set of random variables and a
set of factors to describe the correlation between random variables. We then
run MAP inference on the factor graph using the LP formulation in
\cite{Kschischang2001} and compare the quality of the solutions obtained by
Thetis with a Gibbs sampling-based approach~\cite{Zhang2013}. We follow the
LP-rounding algorithm in~\cite{ravikumar:2010:jmlr} to solve the MAP estimation
problem.  For entity linking, we use the TAC-KBP 2010
benchmark\footnote{\url{http://nlp.cs.qc.cuny.edu/kbp/2010/}}. The input
graphical model has 12K boolean random variables and 17K factors. For text
chunking, we use the CoNLL 2000 shared
task\footnote{\url{http://www.cnts.ua.ac.be/conll2000/chunking/}}. The factor
graph contained 47K categorical random variables (with domain size 23) and 100K
factors. We use the training sets provided by TAC-KBP 2010 and CoNLL 2000
respectively.  We evaluate the quality of both approaches using the official
evaluation scripts and evaluation data sets provided by each challenge.
Figure~\ref{fig:high:level} contains a description of the three relevant
quality metrics, precision (P), recall (R) and F1-scores.
Figure~\ref{fig:high:level} demonstrates that our algorithm produces solutions
of quality comparable with state-of-the-art approaches for these graph analysis
tasks.

\begin{figure}
\footnotesize
\centering
\begin{tabular}{ c c | c c | c  c  c  c| c c c c } 
\multicolumn{2}{c}{} & \multicolumn{2}{c}{}   & \multicolumn{4}{c}{Thetis}
& \multicolumn{4}{c}{Gibbs Sampling} \\ \hline 
Task  &  Formulation & PV & NNZ &  {\bf P} & {\bf R}  & {\bf F1} & {Rank}
&  {\bf P} & {\bf R}  & {\bf F1} & {Rank}  \\ \hline
CoNLL  & Skip-chain CRF & 25M & 51M &  .87 & .90 & .89 & 10/13 & .86 & .90
& .88 & 10/13  \\
TAC-KBP& Factor graph & 62K & 115K&  .79 & .79 & .79 & 6/17 & .80 & .80 & .80
& 6/17 \\
\hline
\end{tabular}
\caption{Solution quality of our LP-rounding approach on two tasks. PV is the
number of primal variables and NNZ is the number of non-zeros in the constraint
matrix of the LP in standard form.  The {\it rank} indicates where we would
been have placed, had we participated in the competition.}
\label{fig:high:level}
\vspace{-0.5cm}
\end{figure}

\subsection{How does our proposed approach compare to a state-of-the-art
commercial solver?} \label{sec:expt_cplex}

We conducted numerical experiments on three different combinatorial problems
that commonly arise in graph analysis tasks in machine learning: vertex cover,
independent set, and multiway cuts. For each problem, we compared the
performance of our LP solver against the LP and IP solvers of \cplex (v12.5)
(denoted as \cplex-LP and \cplex-IP respectively). The two main goals of this
experiment are to: (1) compare the quality of the integral solutions obtained
using LP-rounding with the integral solutions from \cplex-IP and (2) compare
wall-clock times required by Thetis and \cplex-LP to solve the LPs for the
purpose of LP-rounding.

\paragraph*{Datasets.} 

Our tasks are based on two families of graphs. The first family of
instances ({\it frb59-26-1} to {\it frb59-26-5}) was obtained from
Bhoslib\footnote{\url{http://www.nlsde.buaa.edu.cn/~kexu/benchmarks/graph-benchmarks.htm}}
(Benchmark with Hidden Optimum Solutions); they are considered
difficult problems~\cite{Xu2006}. The instances in this family are similar; the
first is reported in the figures of this section, while the remainder appear in
Appendix~\ref{app:experiments}. The second family of instances are social
networking graphs obtained from the Stanford Network Analysis Platform
(SNAP)\footnote{\url{http://snap.stanford.edu/}}.

\paragraph*{System Setup.}
Thetis was implemented using a combination of C++ (for Algorithm~\ref{alg:SCD})
and Matlab (for the augmented Lagrangian framework). Our implementation of the
augmented Lagrangian framework was based on \cite{Eckstein2010}. All
experiments were run on a 4 Intel Xeon E7-4450 (40 cores @ 2Ghz) with 256GB of
RAM running Linux 3.8.4 with a 15-disk RAID0.  \cplex used 32 (of the 40) cores
available in the machine, and for consistency, our implementation was also
restricted to 32 cores. \cplex implements presolve procedures that detect
redundancy, and substitute and eliminate variables to obtain equivalent,
smaller LPs.  Since the aim of this experiment is compare the algorithms used
to solve LPs, we ran both \cplex-LP and Thetis on the reduced LPs generated by
the presolve procedure of \cplex-LP.  Both \cplex-LP and Thetis were run to a
tolerance of $\epsilon=0.1$.  Additional experiments with \cplex-LP run using
its default tolerance options are reported in Appendix~\ref{app:experiments}.
We used the barrier optimizer while running \cplex-LP.  All codes were provided
with a time limit of 3600 seconds excluding the time taken for preprocessing as
well as the runtime of the rounding algorithms that generate integral solutions
from fractional solutions. 

\paragraph*{Tasks.} 
We solved the vertex cover problem using the approximation algorithm
described in Section~\ref{sec:vc:rounding}.  We solved the maximum
independent set problem using a variant of the $es + o(s)$-factor
approximation in \cite{Bansal2012} where $s$ is the maximum degree of
a node in the graph (see Appendix~\ref{app:roundings} for details).
For the multiway-cut problem (with $k=3$) we used the $3/2 -
1/k$-factor approximation algorithm described in \cite{Vazirani}.  The
details of the transformation from approximate infeasible solutions to
feasible solutions are provided in Appendix~\ref{app:roundings}. Since
the rounding schemes for maximum-independent set and multiway-cut are
randomized, we chose the best feasible integral solution from 10
repetitions.

\begin{figure}[h]
\centering
\footnotesize
\begin{tabular}{|c||rr|rr||rr|rr||rr|rr|} 
\hline
& \multicolumn{8}{c|}{{\bf Minimization problems}} &  \multicolumn{4}{c|}{{\bf Maximization problems}} \\ \cline{2-13}
 Instance& \multicolumn{4}{c}{{\bf VC}} & \multicolumn{4}{c | }{{\bf MC}}
 & \multicolumn{4}{c|}{{\bf MIS}} \\ \cline{2-13}
 & PV & NNZ & S & Q & PV & NNZ & S & Q  & PV & NNZ & S & Q \\ \cline{1-13}
frb59-26-1 & 0.12 & 0.37 & 2.8 & 1.04 & 0.75 & 3.02 & 53.3 & 1.01 & 0.12 & 0.38 & 5.3 & 0.36\\
Amazon & 0.39 & 1.17 & 8.4 & 1.23 & 5.89 & 23.2 & - & 0.42 & 0.39 & 1.17 & 7.4 & 0.82\\
DBLP & 0.37 & 1.13 & 8.3 & 1.25 & 6.61 & 26.1 & - & 0.33 & 0.37 & 1.13 & 8.5 & 0.88\\
Google+ & 0.71 & 2.14 & 9.0 & 1.21 & 9.24 & 36.8 & - & 0.83 & 0.71 & 2.14 & 10.2 & 0.82\\
\hline
\end{tabular}
\caption{Summary of wall-clock speedup (in comparison with \cplex-LP) and
solution quality (in comparison with \cplex-IP) of Thetis on three graph
analysis problems. Each code is run with a time limit of one hour and
parallelized over 32 cores, with `-' indicating that the code reached the time
limit. PV is the number of primal variables while NNZ is the number of nonzeros
in the constraint matrix of the LP in standard form (both in millions). S is
the speedup, defined as the time taken by Cplex-LP divided by the time taken by
Thetis. Q is the ratio of the solution objective obtained by Thetis to that
reported by \cplex-IP.  For minimization problems ({\bf VC} and {\bf MC}) lower
Q is better; for maximization problems ({\bf MIS}) higher Q is better. For MC,
a value of $Q<1$ indicates that Thetis found a better solution than \cplex-IP
found within the time limit.  }
\label{tab:ip_summary}
\end{figure}

\begin{figure}[h]
\footnotesize
\centering
\setlength{\tabcolsep}{3pt}
\begin{tabular}{ | c | c c c | c c c | c c c | }
\hline
{\bf VC} & \multicolumn{3}{c|}{ \cplex IP} & \multicolumn{3}{c|}{ \cplex LP} &
\multicolumn{3}{c|}{ Thetis } \\ \cline{2-10}
(min) & t (secs) & BFS & Gap (\%) & t (secs) & LP & RSol & t (secs) & LP & RSol \\ \hline
frb59-26-1 & - & 1475 & 0.67 & 2.48 & 767 & 1534 & 0.88 & 959.7 & 1532\\
Amazon & 85.5 & 1.60\e{5} & -& 24.8 & 1.50\e{5} & 2.04\e{5} & 2.97 & 1.50\e{5} & 1.97\e{5}\\
DBLP & 22.1 & 1.65\e{5} & - & 22.3 & 1.42\e{5} & 2.08\e{5} & 2.70 & 1.42\e{5} & 2.06\e{5}\\
Google+ & - & 1.06\e{5} & 0.01 & 40.1 & 1.00\e{5} & 1.31\e{5} & 4.47 & 1.00\e{5} & 1.27\e{5}\\
\hline
{\bf MC} & \multicolumn{3}{c|}{ \cplex IP} & \multicolumn{3}{c|}{ \cplex LP}
& \multicolumn{3}{c|}{ Thetis} \\ \cline{2-10}
(min) & t (secs) & BFS & Gap (\%) & t (secs) & LP & RSol & t (secs) & LP & RSol \\ \hline
frb59-26-1 & 72.3 & 346 & - & 312.2 & 346 & 346 & 5.86 & 352.3 & 349\\
Amazon & - & 12 & NA & - & - & - & 55.8 & 7.28 & 5\\
DBLP & - & 15 & NA& - & - & - & 63.8 & 11.7 & 5\\
Google+ & - & 6 & NA & - & - & - & 109.9 & 5.84 & 5\\
\hline
{\bf MIS} & \multicolumn{3}{c|}{ \cplex IP} & \multicolumn{3}{c|}{ \cplex LP} &
\multicolumn{3}{c|}{ Thetis} \\ \cline{2-10}
(max) & t (secs) & BFS & Gap (\%) & t (secs) & LP & RSol & t (secs) & LP & RSol \\ \hline
frb59-26-1 & - & 50 & 18.0 & 4.65 & 767 & 15 & 0.88 & 447.7 & 18\\
Amazon & 35.4 & 1.75\e{5} & - & 23.0 & 1.85\e{5} & 1.56\e{5} & 3.09 & 1.73\e{5} & 1.43\e{5}\\
DBLP & 17.3 & 1.52\e{5} & - & 23.2 & 1.75\e{5} & 1.41\e{5} & 2.72 & 1.66\e{5} & 1.34\e{5}\\
Google+ & - & 1.06\e{5} & - & 44.5 & 1.11\e{5} & 9.39\e{4} & 4.37 & 1.00\e{5} & 8.67\e{4}\\
\hline
\end{tabular}
\caption{Wall-clock time and quality of fractional and integral solutions for
three graph analysis problems using Thetis, \cplex-IP and \cplex-LP. Each code
was given a time limit of one hour, with `-' indicating a timeout.  BFS is the
objective value of the best integer feasible solution found by \cplex-IP.  The
gap is defined as (BFS$-$BB)/BFS where BB is the best known solution bound
found by \cplex-IP within the time limit.  A gap of `-' indicates that the
problem was solved to within $0.01\%$ accuracy and NA indicates that \cplex-IP
was unable to find a valid solution bound.  LP is the objective value of the LP
solution, and RSol is objective value of the rounded solution.}
\label{tab:ip_details}
\vspace{-0.5cm}
\end{figure}

\paragraph*{Results.} 
The results are summarized in Figure~\ref{tab:ip_summary}, with additional
details in Figure~\ref{tab:ip_details}. We discuss the results for the vertex
cover problem. On the Bhoslib instances, the integral solutions from Thetis
were within 4\% of the documented optimal solutions. In comparison,  \cplex-IP
produced integral solutions that were within 1\% of the documented optimal
solutions, but required an hour for each of the instances.  Although the LP
solutions obtained by Thetis were less accurate than those obtained by
\cplex-LP, the rounded solutions from Thetis and \cplex-LP are almost exactly
the same. In summary, the LP-rounding approaches using Thetis and \cplex-LP
obtain integral solutions of comparable quality with \cplex-IP --- but Thetis
is about three times faster than \cplex-LP.  

We observed a similar trend on the large social networking graphs. We were able
to recover integral solutions of comparable quality to \cplex-IP, but seven to
eight times faster than using LP-rounding with \cplex-LP. We make two
additional observations.  The difference between the optimal fractional and
integral solutions for these instances is much smaller than frb59-26-1. We
recorded unpredictable performance of \cplex-IP on large instances. Notably,
\cplex-IP was able to find the optimal solution for the {\it Amazon} and {\it
DBLP} instances, but timed out on {\it Google+}, which is of comparable size.
On some instances, \cplex-IP outperformed even \cplex-LP in wall clock time,
due to specialized presolve strategies.

\section{Conclusion}

We described Thetis, an LP rounding scheme based on an approximate solver for
LP relaxations of combinatorial problems.  We derived worst-case runtime and
solution quality bounds for our scheme, and demonstrated that our approach was
faster than an alternative based on a state-of-the-art LP solver, while
producing rounded solutions of comparable quality.

\section*{Acknowledgements}
SS is generously supported by ONR award N000141310129.  JL is generously
supported in part by NSF awards DMS-0914524 and DMS-1216318 and ONR award
N000141310129. CR's work on this project is generously supported by NSF CAREER
award under IIS-1353606, NSF award under CCF-1356918, the ONR under awards
N000141210041 and  N000141310129, a Sloan Research Fellowship, and gifts from
Oracle and Google.  SJW is generously supported in part by NSF awards
DMS-0914524 and DMS-1216318, ONR award N000141310129, DOE award DE-SC0002283,
and Subcontract 3F-30222 from Argonne National Laboratory.  Any 
recommendations, findings or opinions expressed in this work are those of the
authors and do not necessarily reflect the views of any of the above sponsors.

\newpage
\small
\bibliography{lp}

\begin{thebibliography}{10}

\bibitem{Sridhar2013NIPS}
S.~Sridhar, V.~Bittorf, J.~Liu, C.~Zhang, C.~R{\'e}, and S.~J. Wright, ``{An
  approximate, efficient solver for LP rounding},'' in {\em Advances in Neural
  Information Processing Systems 26}, 2013.

\bibitem{Koval:1976:lr}
V.~Koval and M.~Schlesinger, ``Two-dimensional programming in image analysis
  problems,'' {\em USSR Academy of Science, Automatics and Telemechanics},
  vol.~8, pp.~149--168, 1976.

\bibitem{ravikumar:2010:jmlr}
P.~Ravikumar, A.~Agarwal, and M.~J. Wainwright, ``Message-passing for
  graph-structured linear programs: Proximal methods and rounding schemes,''
  {\em The Journal of Machine Learning Research}, vol.~11, pp.~1043--1080,
  2010.

\bibitem{Vazirani}
V.~V. Vazirani, {\em Approximation Algorithms}.
\newblock Springer, 2004.

\bibitem{Lempitsky:2007:cvpr}
V.~Lempitsky and Y.~Boykov, ``Global optimization for shape fitting,'' in {\em
  IEEE Conference on Computer Vision and Pattern Recognition (CVPR '07)},
  pp.~1--8, IEEE, 2007.

\bibitem{Ren94b}
J.~Renegar, ``Some perturbation theory for linear programming,'' {\em
  Mathenatical Programming, Series A}, vol.~65, pp.~73--92, 1994.

\bibitem{Sanghavi2007a}
S.~Sanghavi, D.~Malioutov, and A.~S. Willsky, ``Linear programming analysis of
  loopy belief propagation for weighted matching,'' in {\em Advances in Neural
  Information Processing Systems}, pp.~1273--1280, 2007.

\bibitem{Niu2011}
F.~Niu, B.~Recht, C.~R{\'e}, and S.~J. Wright, ``Hogwild!: A lock-free approach
  to parallelizing stochastic gradient descent,'' {\em arXiv preprint
  arXiv:1106.5730}, 2011.

\bibitem{LiuW13b}
J.~Liu, S.~J. Wright, C.~{R\'e}, and V.~Bittorf, ``An asynchronous parallel
  stochastic coordinate descent algorithm,'' tech. rep., University of
  Wisconsin-Madison, October 2013.

\bibitem{Makari:2013:VLDB}
F.~Manshadi, B.~Awerbuch, R.~Gemulla, R.~Khandekar, J.~Mestre, and M.~Sozio,
  ``A distributed algorithm for large-scale generalized matching,'' {\em
  Proceedings of the VLDB Endowment}, 2013.

\bibitem{Srinivasan1999}
A.~Srinivasan, ``Improved approximation guarantees for packing and covering
  integer programs,'' {\em SIAM Journal on Computing}, vol.~29, no.~2,
  pp.~648--670, 1999.

\bibitem{Bien2009}
J.~Bien and R.~Tibshirani, ``Classification by set cover: The prototype vector
  machine,'' {\em arXiv preprint arXiv:0908.2284}, 2009.

\bibitem{Wu2012}
Z.~Wu, A.~Thangali, S.~Sclaroff, and M.~Betke, ``Coupling detection and data
  association for multiple object tracking,'' in {\em Computer Vision and
  Pattern Recognition (CVPR), 2012 IEEE Conference on}, pp.~1948--1955, IEEE,
  2012.

\bibitem{Bansal2012}
N.~Bansal, N.~Korula, V.~Nagarajan, and A.~Srinivasan, ``Solving packing
  integer programs via randomized rounding with alterations.,'' {\em Theory of
  Computing}, vol.~8, no.~1, pp.~533--565, 2012.

\bibitem{Kschischang2001}
F.~R. Kschischang, B.~J. Frey, and H.-A. Loeliger, ``Factor graphs and the
  sum-product algorithm,'' {\em Information Theory, IEEE Transactions on},
  vol.~47, no.~2, pp.~498--519, 2001.

\bibitem{Rabini1998}
G.~C{\u{a}}linescu, H.~Karloff, and Y.~Rabani, ``An improved approximation
  algorithm for multiway cut,'' in {\em Proceedings of the thirtieth annual ACM
  symposium on Theory of Computing}, pp.~48--52, ACM, 1998.

\bibitem{Yuri2004}
Y.~Boykov and V.~Kolmogorov, ``An experimental comparison of min-cut/max-flow
  algorithms for energy minimization in vision,'' {\em IEEE Transactions on
  Pattern Analysis and Machine Intelligence}, vol.~26, pp.~1124--1137, 2004.

\bibitem{Lee2011}
T.~Lee, Z.~Wang, H.~Wang, and S.-w. Hwang, ``Web scale entity resolution using
  relational evidence,'' tech. rep., Microsoft Research, 2011.

\bibitem{Roth:2005:icml}
D.~Roth and W.-t. Yih, ``Integer linear programming inference for conditional
  random fields,'' in {\em Proceedings of the 22nd International Conference on
  Machine Learning}, pp.~736--743, ACM, 2005.

\bibitem{VanGael:2007:ijcai}
J.~Van~Gael and X.~Zhu, ``Correlation clustering for crosslingual link
  detection.,'' in {\em IJCAI}, pp.~1744--1749, 2007.

\bibitem{Ber99}
D.~P. Bertsekas, {\em Nonlinear Programming}.
\newblock Athena Scientific, 1999.

\bibitem{NocW06}
J.~Nocedal and S.~J. Wright, {\em Numerical Optimization}.
\newblock Springer, 2006.

\bibitem{SJW14}
S.~J. Wright, ``Implementing proximal point methods for linear programming,''
  {\em Journal of Optimization Theory and Applications}, vol.~65, no.~3,
  pp.~531--554, 1990.

\bibitem{Zhang2013}
C.~Zhang and C.~R{\'e}, ``{Towards high-throughput Gibbs sampling at scale: A
  study across storage managers},'' in {\em SIGMOD Proceedings}, 2013.

\bibitem{Xu2006}
K.~Xu and W.~Li, ``Many hard examples in exact phase transitions,'' {\em
  Theoretical Computer Science}, vol.~355, no.~3, pp.~291--302, 2006.

\bibitem{Eckstein2010}
J.~Eckstein and P.~J. Silva, ``A practical relative error criterion for
  augmented lagrangians,'' {\em Mathematical Programming}, pp.~1--30, 2010.

\bibitem{Hochbaum1982}
D.~S. Hochbaum, ``Approximation algorithms for the set covering and vertex
  cover problems,'' {\em SIAM Journal on Computing}, vol.~11, no.~3,
  pp.~555--556, 1982.

\end{thebibliography}

\newpage
\normalsize

\appendix

\begin{center}
{\bf\large Supplementary Material}
\end{center}

\section{Perturbation Results}

We discuss here the perturbation results for the quadratic
approximation \eqnok{eq:QP} to the linear program \eqnok{eq:LP}. These
results constitute a proof of Theorem~\ref{th:xh2}.

We note for future reference that the optimality (KKT) conditions for the
primal-dual pair of LPs \eqnok{eq:LP} and \eqnok{eq:LPdual} are 
\beq \label{eq:lp.kkt}
Ax=b, \quad 0 \le c- A^T u \, \perp \, x \ge 0.
\eeq
The QP approximation \eqnok{eq:QP} is equivalent to the following
monotone linear complementarity problem (LCP):
\beq \label{eq:lpb.lcp}
0 \le x \, \perp \, 
F_{\beta}(x) := c -A^T\bar{u} + \beta A^T(Ax-b) + \frac{1}{\beta} (x-\bar{x}).
\eeq


Here we 
rely on Renegar's theory \cite{Ren94b} which requires not only that
primal and dual are both solvable, but also that they are still
solvable after we make arbitrary small perturbations to the data
$(A,b,c)$. This includes cases in which the basis has fewer nonzeros
than there are equality constraints (a situation known as ``primal
degeneracy'').  We assume throughout that $A$ has full row rank $m$.
If $A$ were row rank deficient, then even if the primal-dual pair had
a solution, we would be able to find an arbitrarily small perturbation
that renders the primal infeasible.

In accordance with Renegar, we use $d:=(A,b,c)$ to denote the data for
the problems \eqnok{eq:LP} and \eqnok{eq:LPdual}. We denote by $\Prio$ the
set of data $d$ for which the primal \eqnok{eq:LP} is infeasible, and
define $\Dualo$ analogously for the dual \eqnok{eq:LPdual}. Renegar uses
the ``distance to infeasibility'' to define a condition number for the
primal and dual. Specifically, defining
\beq \label{eq:def.delta}
\delta_P := \frac{\dist(d,\Prio)}{\|d\|}, \quad
\delta_D := \frac{\dist(d,\Dualo)}{\|d\|},
\eeq
the quantities $1/\delta_P$ and $1/\delta_D$ capture the sensitivity
of the optimal objective value for the problem \eqnok{eq:LP} to
perturbations in $b$ and $c$. Note that both $\delta_P$ and $\delta_D$
lie in the interval $[0,1]$.

We assume $\delta_P>0$ and $\delta_D>0$ throughout the analysis below. This
implies that the primal and dual are both feasible, hence by strong duality
both have solutions $x^*$ and $u^*$ (not necessarily unique).

\begin{lemma} \label{lem:xb}
Suppose that $\delta_P>0$ and $\delta_D>0$, and let $x^*$ be any solution of
\eqnok{eq:LP} and $u^*$ be any solution of \eqnok{eq:LPdual}, and define 
\[
C_* := \max(\|x^*-\bar{x}\|,\|u^*-\bar{u}\|).
\]
Then the unique solution $x(\beta)$ of \eqnok{eq:QP} satisfies the following
inequalities:
\begin{align*}
\|Ax(\beta)-b \|  & \le \beta^{-1} \left[\|u^* -\bar{u} \| + 
\sqrt{\|u^*-\bar{u}\|^2 + \|x^* -\bar{x}\|^2} \right] \\
& \le \beta^{-1} (1+\sqrt{2})C_*, \\
\| x(\beta) -\bar{x} \| & \le  \left[ 2 \|u^*-\bar{u}\| \left[ \|u^*-\bar{u}\| + \sqrt{\|u^*-\bar{u}\|^2+\|x^*-\bar{x}\|^2}\right] +
\|x^*-\bar{x}\|^2 \right]^{1/2} \\
 & \le \sqrt{6} C_*.
\end{align*}
\end{lemma}
\begin{proof}
Note that $x^*$ is a feasible point for \eqnok{eq:QP}, so we have by optimality
of $x(\beta)$ that $f_{\beta}(x(\beta)) \le f_{\beta}(x^*)$, that is,
\[
c^T x(\beta) -\bar{u}^T (Ax(\beta)-b) + \frac{\beta}{2} \|A x(\beta) - b \|^2 + \frac{1}{2 \beta} \|x(\beta)-\bar{x} \|^2 \le c^T x^* + \frac{1}{2 \beta} \|x^*-\bar{x} \|^2,
\]
and thus
\[
\frac{\beta}{2} \|A x(\beta) - b \|^2 + \frac{1}{2 \beta} \|x(\beta) -\bar{x}\|^2 \le
c^T (x^*-x(\beta) ) + \bar{u}^T(Ax(\beta)-b) + \frac{1}{2 \beta} \|x^* -\bar{x} \|^2.
\]
Note from $x(\beta) \ge 0$ and \eqnok{eq:lp.kkt} that 
\[
0 \le  x(\beta)^T (c-A^T u^*) \;\; \Rightarrow \;\; -c^T x(\beta) \le -(u^*)^T A x(\beta).
\]
We also have from \eqnok{eq:lp.kkt}  that $c^Tx^* = (u^*)^TAx^*$.  By
combining these observations, we obtain
\beq \label{eq:xb1}
\frac{\beta}{2} \|A x(\beta) - b \|^2 + \frac{1}{2 \beta} \|x(\beta) -\bar{x} \|^2 \le
(u^*-\bar{u})^T A (x^*-x(\beta)) +  \frac{1}{2 \beta} \|x^* -\bar{x}\|^2.
\eeq
By dropping the second term on the left-hand side of this expression,
multiplying by $\beta$, and using Cauchy-Schwartz and $Ax^*=b$, we obtain
\[
\frac{\beta^2}{2} \| Ax(\beta)-b \|^2 \le  \| u^* -\bar{u} \| \beta \| Ax(\beta) - b \| +
\frac{1}{2} \| x^* -\bar{x} \|^2.
\]
Denoting $e_{\beta} := \beta \| Ax(\beta)-b \|$, this inequality
reduces to the condition
\[
\frac12 e_{\beta}^2 - \| u^*-\bar{u} \| e_{\beta}  - \frac12 \|x^*-\bar{x}\|^2 \le 0.
\]
Solving this quadratic for $e_{\beta}$, we obtain 
\[
e_{\beta} \le \|u^* -\bar{u} \| + \sqrt{\|u^*-\bar{u}\|^2 + \|x^* -\bar{x}\|^2},
\]
proving the first claim.

For the second claim, we return to \eqnok{eq:xb1}, dropping the first
term on the left-hand side, to obtain
\[
\frac{1}{2 \beta} \| x(\beta) -\bar{x} \|^2 \le \|u^*-\bar{u} \| \|Ax(\beta)-b \| + \frac{1}{2 \beta} \|x^*-\bar{x}\|^2.
\]
By substituting the bound on $\|Ax(\beta)-b\|$ just derived,
multiplying by $2 \beta$ and taking the square root, we obtain the
result.
\end{proof}

Fixing $\beta$ and $x(\beta)$, we now consider the following perturbed
linear program
\beq \label{eq:lp.beta}
\min \, c_{\beta}^T x \;\; \mbox{s.t.} \;\; Ax=b_{\beta}, \;\; x \ge 0,
\eeq
and its dual
\beq \label{eq:lp.beta.d}
\max \, b_{\beta}^Tu \;\; \mbox{s.t.} \;\; A^Tu \le c_{\beta},
\eeq
where
\[
b_{\beta} := Ax(\beta), \quad c_{\beta} := c + \frac{1}{\beta} (x(\beta)-\bar{x}).
\]
From Lemma~\ref{lem:xb}, we have
\beq \label{eq:bounds.delta}
\| b-b_{\beta} \| \le \frac{1}{\beta} (1+\sqrt{2}) C_* \le \frac{2.5 C_*}{\beta}, \quad
\|c-c_{\beta} \| \le \frac{1}{\beta} \sqrt{6} C_* \le \frac{2.5 C_*}{\beta}.
\eeq
KKT conditions for \eqnok{eq:lp.beta}, \eqnok{eq:lp.beta.d} are
\[
0 \le \hat{x} \perp c + \frac{1}{\beta} - A^T \hat{u} \ge 0, \quad
A\hat{x}=A x(\beta).
\]
It is easy to check, by comparing with \eqnok{eq:lpb.lcp}, that these
conditions are satisfied by
\[
\hat{x}=x(\beta), \quad \hat{u} = \bar{u}-\beta A(x(\beta)-x^*).
\]
Hence $\hat{x}=x(\beta)$ is a solution of \eqnok{eq:lp.beta}. There
may be other solutions, but they will have the same objective value,
of course.

We now use the following result, which follows immediately from
\cite[Theorem~1, part (5)]{Ren94b}.\footnote{Note that Renegar appears
  to use a different formulation for the linear program, namely $Ax
  \le b$ rather than $Ax=b$. However, his inequality represents a
  complete ordering with respect to a closed convex cone $C_Y$, and
  when we set $C_Y = \{0\}$, we recover $Ax=b$.}
\begin{theorem} \label{th:renegar}
Let $d = (A,b,c)$ be the data defining the primal-dual pair
\eqnok{eq:LP} and \eqnok{eq:LPdual}, and suppose that $\delta_P$ and
$\delta_D$ defined by \eqnok{eq:def.delta} are both positive. Consider
the following perturbation applied to the $b$ and $c$ components:
$\Delta d := (0,\Delta b, \Delta c)$, and assume that 
\[
\frac{\| \Delta d\|}{\|d\|} \le \delta_P, \quad
\frac{\| \Delta d\|}{\|d\|} \le \delta_D.
\]
Then, denoting the solution of \eqnok{eq:LP} by $x^*$ and the solution
of the linear program with perturbed data $d + \Delta d$ by $x_{\Delta}^*$, we have
\[
| c^Tx^* - (c+\Delta c)^T x_{\Delta}^* |  \le 
\frac{\| \Delta b\|}{\delta_D} \frac{\|c\|+\| \Delta c \|}{\dist(d,\Prio) - \| \Delta d\|} +
\frac{\| \Delta c\|}{\delta_P} \frac{\|b\|+\|\Delta b\|}{\dist (d,\Dualo) - \| \Delta d\|}.
\]
\end{theorem}

Our main theorem is obtained by applying this result with the
perturbations
\beq \label{eq:delbc}
\Delta b := b_{\beta}-b = Ax(\beta)-b, \quad
\Delta c := c_{\beta}-c = \frac{1}{\beta} (x(\beta)-\bar{x}).
\eeq
We have the following result.
\begin{theorem} \label{th:mp}
Suppose that
\[
\beta \ge \bar{\beta} := \frac{10 C_*}{\|d\| \min(\delta_P, \delta_D)}.
\]
We then have the following bound on the difference between the optimal
values of \eqnok{eq:LP} and \eqnok{eq:lp.beta}:
\[
| c^Tx^* - c_{\beta}^T x(\beta) | \le \frac{1}{\beta} \frac{25 C_*}{2 \delta_P \delta_D}.
\]
\end{theorem}
\begin{proof}
Note first that from \eqnok{eq:bounds.delta} and  
\[
\| \Delta d \| \le \| \Delta b \| + \| \Delta c \|  \le \frac{5 C_*}{\beta}.
\]
From our assumption on $\beta$, we have
\[
\frac{\| \Delta d\|}{\|d\|} \le \frac{5C_*}{\beta \|d\|} \le \frac12 \min(\delta_P,\delta_D),
\]
so that the assumptions of Theorem~\ref{th:renegar} are satisfied. We have moreover
from the definitions \eqnok{eq:def.delta} that
\[
\dist(d,\Prio) - \| \Delta d\| 
= \| d\| \left[ \delta_P - \frac{\|\Delta d\|}{\|d\|} \right]
\ge \frac12 \| d \| \delta_P,
\]
and similarly $\dist(d,\Dualo) \ge (1/2) \|d\| \delta_D$. By
substituting into the inequality of Theorem~\ref{th:renegar}, and
using the bounds just derived together with \eqnok{eq:bounds.delta},
we obtain
\[
| c^* x^* - c_{\beta}^Tx(\beta) | \le 
\frac{2.5 \beta^{-1} C_*}{\delta_D} 
\frac{(\|c\| + 2.5 \beta^{-1} C_*)}{.5 \|d\| \delta_P} +
\frac{2.5 \beta^{-1} C_*}{\delta_P} 
\frac{(\|b\| + 2.5 \beta^{-1} C_*)}{.5 \|d\| \delta_D}.
\]
Since
\[
\|c\| \le \|d\|, \;\; \|b\| \le \|d\|, \;\; \frac{2.5 C_*}{\beta} \le \frac{1}{4}
\min(\delta_P,\delta_D) \|d\| \le \frac{1}{4} \|d\|,
\]
we have
\[
| c^* x^* - c_{\beta}^Tx(\beta) | \le \frac{2.5 \beta^{-1} C_* (2.5)
  \|d\|}{(1/2) \|d\| \delta_P \delta_D} = \frac{1}{\beta} \frac{25
  C_*}{2 \delta_P \delta_D},
\]
completing the proof.
\end{proof}

The following corollary is almost immediate.
\begin{corollary} \label{co:mp}
Suppose  the conditions of Theorem~\ref{th:mp} are satisfied. Then
\[
| c^Tx^* - c^T x(\beta) | \le \frac{1}{\beta} \left[ \frac{25 C_*}{2
    \delta_P \delta_D} + 6 C_*^2 + \sqrt{6} \| \bar{x} \| C_* \right].
\]
\end{corollary}
\begin{proof}
We have from the definition of $c_{\beta}$  that
\begin{align*}
|c^T x^* - c^T x(\beta)| & \le |c^Tx^* - c_{\beta}^Tx(\beta)| +
\frac{1}{\beta} x(\beta)^T(x(\beta)-\bar{x})   \\
&=  |c^Tx^* - c_{\beta}^Tx(\beta)|  + \frac{1}{\beta} \| x(\beta)-\bar{x}\|^2 + 
\frac{1}{\beta} \bar{x}^T(x(\beta)-\bar{x}) \\
& \le \frac{1}{\beta}  \left[ \frac{25 C_*}{2\delta_P \delta_D} + 
6C_*^2 + \sqrt{6} \|\bar{x}\|  C_* \right].
\end{align*}
where the final inequality follow from Lemma~\ref{lem:xb} and Theorem~\ref{th:mp}.
\end{proof}

\section{Details of Rounding Schemes}

In this section, we provide details of known LP-rounding schemes for covering,
packing and multiway-cut problems. (Vazirani~\cite{Vazirani} provides a
comprehensive survey on the theory and algorithms for LP-rounding.). We then
discuss how these algorithms can be extended to round $(\epsilon, \delta)$
optimal solutions.

\subsection{Set Cover}
Given a universe $U$ with $N$ elements, a collection of subsets $\cS =
\{S_1, S_2 \ldots S_k\}$ each associated with a positive cost function
$c: S \rightarrow \bRp$. In the set cover problem, we must identify a
minimum cost sub-collection of sets $S$ that covers all elements in
$U$. The set cover problem can be formulated as the following IP:
\begin{equation}
\min \sum_{s \in \cS} c_s x_s \spc{3} \mbox{subject to} \spc{3}
\sum_{s: a \in s} x_{s} \geq 1 \ \ \forall a \in U, \ x_s \in \{0,1\}
\ \forall s \in \cS.
\label{model:set_cover} 
\end{equation}
A simple way to convert a solution $x_s^*$ of the LP relaxation to an integral
solution is to pick all sets $x_s$ where $x_s^* > 1/f$, where $f$ is a bound on
the maximum number of sets in which a single element is present. Such an
algorithm achieves an $f$-factor approximation \cite{Hochbaum1982}. An
alternative approximation scheme is a randomized scheme due to
\cite{Srinivasan1999}. In this scheme, we put $s \in \cS$ into the set cover
with probability equal to the optimal fractional solution $x_s^*$. In
expectation, this approximation scheme is a $O(\log N)$-factor approximation,
and is a valid set cover with probability $1/2$.


\subsection{Set Packing}
Using the same notation for $U$, $N$, $\cS$, and $x_s, \ \forall s \in
\cS$ as above, the set packing problem is to identify the lowest cost
collection of mutually disjoint sets. It can be formulated as the
following IP:
\begin{equation}
\max \sum_{s \in \cS} c_s x_s \spc{3} \mbox{subject to} \spc{3} \sum_{s: a \in
s} w_{a,s}x_{s} \leq 1 \  \  \forall a \in U, \ x_s \in \{0,1\} \ \forall s \in
\cS,  \label{model:set_packing} 
\end{equation}
where $w_{a,s}$ is the weight of element $a \in U$ in set $s \in \cS$.

Bansal et al. \cite{Bansal2012} proposed an $ek + o(k)$-factor approximation
(see Algorithm \ref{alg:set_packing_rounding_2}) for the special case of
$k$-column sparse set packing where the maximum number of sets containing each
element is at most $k$. They use the following stronger formulation of the set
packing problem:
\begin{alignat}{2}\label{model:strong_set_packing}
& \max \sum_{s \in \cS} c_s x_s  \\ 
\mbox{subject to } \spc{3} & \sum_{s: a \in s} w_{a,s} x_s \leq 1 \spc{5} &  \forall a \in U, \nonumber \\
& \sum_{a \in B(s)} x_{s} \leq 1 & \forall a \in U, \nonumber \\
& x_{s} \in \{0,1\} & \forall s \in \cS, \nonumber
\end{alignat}
where $w_{a,s} = 1$ if the element $a \in U$ is present in set $s \in
S$, $c_s$ is the cost of set $s \in S$ and $B(s) := \{a \in U |
w_{a,s} > 1/2 \}.$

\begin{algorithm}
\label{alg:set_packing_rounding_2}
\begin{algorithmic}[1]
\STATE Find any feasible solution $\hat{x}$ to the LP relaxation of \eqref{model:strong_set_packing}.
\STATE Choose set $s \in \cS$ with probability ${\hat{x}_s}/{(k
  \theta)}$. Let $\cC \subseteq \cS$ denote the chosen sets.
\STATE For each set $s \in \cC$ and element $a \in U$, let $E_{a,s}$
denote the event that the sets $\{s_2 \in \cC: w_{a,s_2} > w_{a,s} \}$
have a total weight (with respect to element $a$) exceeding 1. Mark
$s$ for deletion if $E_{a,s}$ occurs for any $a \in s$.
\STATE Delete all sets from $s \in \cC$ that are marked for deletion.
\end{algorithmic}
\caption{A $ek + o(k)$-factor randomized LP-rounding algorithm for set packing} 
\end{algorithm}

\subsection{Multiway-Cuts}
Given a graph $G(V,E)$ and a set of terminals $V_1, V_2, \ldots V_k$, a $k$-way
cut partitions the set of vertices $V$ into $k$ mutually disjoint sets. The
cost of the $k$-way cut is the sum of the costs of all the edges that run
across the partitions. A $k$-way cut of minimum cost is the solution to the
following problem: 
\begin{align}
\label{model:lp_multiway_cut}
& \min \frac{1}{2} \sum_{u,v \in E} c_{u,v} \sum_{i=1}^{k}| x^i_u - x^i_v |   \\
\mbox{subject to} & \spc{5} x_v \in \Delta_k \spc{5} \forall v \in V \nonumber \\
& x_v \in \{0,1\}^k \spc{5} \forall v \in V, \nonumber
\end{align}
where $\Delta_k := \{x \in \mathbb{R}^k: \sum_{i=1}^{k} x_i = 1, \ x
\geq 0 \}$ is the set of simplex constraints in $k$
dimensions. Although it might appear that the formulation in
\eqref{model:lp_multiway_cut} is non-linear, one can easily linearize
\eqref{model:lp_multiway_cut} to 
\begin{alignat*}{2}
& \min \frac{1}{2} \sum_{u,v \in E} c_{u,v} \sum_{i=1}^{k}  x^{i}_{uv}   \nonumber \\
\mbox{subject to} & \spc{5} x_v \in \Delta_k && \forall v \in V \nonumber \\
& x^i_{uv} \geq x_v^i - x_u^i  \spc{5} && \forall u,v \in E, i \in \{1 \ldots k\} \nonumber \\
& x^i_{uv} \geq x_u^i - x_v^i  \spc{5} && \forall u,v \in E, i \in \{1 \ldots k\} \nonumber \\
& x^i_{uv} \in [0,1] && \forall u,v \in E, i \in \{1 \ldots k\} \nonumber \\
& x^i_{v} \in \{0,1\} && \forall v \in V, i \in \{1 \ldots k\} \nonumber
\end{alignat*}
There is a $3/2 - 1/k$ factor approximation for multiway-cut using the
region-growing algorithm due to \cite{Rabini1998}. The details of the algorithm
are laid out in \cite[Algorithm~19.4]{Vazirani}.

\section{Rounding Infeasible Solutions}
\label{app:roundings}

In this section, we briefly describe how we can extend known
LP-rounding algorithms to infeasible $(\epsilon,\delta)$-approximate
solutions.  We discuss how one can go from an $(\epsilon,
\delta)$-approximate solution to a feasible $(0,
f(\epsilon,\delta))$-approximate solution, for some positive function
$f(\cdot,\cdot)$. The arguments in this section are based on simple ideas of
scaling and projection.

As is the case in the main manuscript, we illustrate our approach
using vertex cover. Let $\hat{x}$ be an
$(\epsilon,\delta)$-approximate solution of the following vertex cover
LP:
\[ 
\min_{x \in [0,1]^{n}} 1^{T}x \quad \mbox{subject to} \;\;  x_i + x_j \geq 1 \text{
  for } (i,j) \in E,
\]
so that in particular, $x_i \in [0,1]$ for all $i$, and $x_i+x_j \ge
1-\epsilon$ for all $(i,j)\in E$.  We claim that the point
\[
z := \Pi_{[0,1]^{n}}(x/(1-\epsilon))
\]
is a $(0,\delta/(1-\varepsilon))$-approximate solution.  To check
feasibility, suppose for contradiction that $z_i+z_j <1$ for some
$(i,j) \in E$. We thus have $z_i<1$ and $z_j<1$, so that
$z_i=x_i/(1-\epsilon)$ and $z_j=x_j/(1-\epsilon)$. Therefore, $z_i+z_j
= (x_i+x_j)/(1-\epsilon) \ge 1$, a contradiction.

\subsection{Rounding for Coverings}

We consider a covering program $P=(A,b,c)$ with positive integer data,
that is, $(A,b,c) \geq 0$ and $A \in \mathbb{Z}^{m \times n}$, $b \in
\mathbb{Z}^{m}$, and $c \in \mathbb{Z}^{n}$. Suppose that there are
also $[0,1]$ bound constraints on each component of $x$. The problem
formulation is as follows:
\[ 
\min_{x \in [0,1]^{n}} c^{T}x \quad \mbox{subject to} \; Ax \geq b. 
\qquad\qquad \mbox{[$P(A,b,c)$]}
\]
To obtain a formulation closer to the standard form \eqnok{eq:LP}, we
can introduce slack variables and write
\[ 
\min_{x \in [0,1]^{n}, z \in [0,\infty)^m} c^{T}x \quad
\mbox{subject to} \; Ax - z = b, \; z \geq 0.  
\]
We can always set $z = \max \set{Ax - b, 0}$ to translate between
feasible solutions of the two programs.

The following quantity $q(P)$ defines a minimum infeasibility measure
over all infeasible, integral solutions to $P$:
\[
 q(P) = \min_{j=1,\dots,m} \min_{x \in \set{0,1}^{n} : A_{j\cdot} x < b_j}
 b_j - A_{j\cdot} x,
\]
where $A_{j \cdot}$ denotes the $j$th row of $A$.  Notice for $q(P)
\geq 1$ for any non-trivial covering program $P$, by integrality
alone.

\begin{lemma}
Let $P$ be a covering program with a nonempty solution set. Let
$\hat{x}$ be an $(\epsilon, \delta)$-approximate solution of $P$,
and suppose that $\epsilon/q(P) \le 1$.  Then there is a
$(0,{\delta}/{(1-\alpha)})$-approximate solution $\tilde{x}$ defined
as
\[  \tilde{x} = \Pi_{[0,1]^{n}}( (1-\alpha)^{-1} \hat{x}), \]
where $\alpha  \in [\epsilon/q(P),1)$.
\end{lemma}
\begin{proof}
We first show that $\tilde{x}$ is feasible. Without loss of
generality, assume that $z_j = \max(A_{j\cdot} \hat{x} - b_j, 0)$ for
$j=1,\dots,m$. Since $\hat{x}$ is a $(\epsilon,\delta)$ solution, we
have $\norm{A \hat{x} - z - b}_{\infty} \leq \epsilon$. With $z$
defined as in our formula, this bound implies that 
\begin{equation}
\label{eq:feas:bound}
A \hat{x}=b \ge -\epsilon \bfone,
\end{equation}
where $\bfone$ is the all-ones vector in $\R^n$. After scaling by
$\hat{x}$ by $(1- \alpha)^{-1}$, some components may exceed
$1$. Hence, we partition the indices into two sets $\Omega_{1} =
\setof{i}{\hat{x}_i \geq 1 - \alpha}$ and $\Omega_{< 1} =
\{1,2,\dotsc,n\} \setminus \Omega_{1}$. For any $\Omega \subseteq
   [n]$, we define the following projection operator:
\[
\pi_{\Omega}(x) := \begin{cases} x_i & \; \mbox{if $i \in \Omega$} \\
0 & \; \mbox{otherwise}.
\end{cases}
\]
We can then write $\tilde{x}$ as follows:
\[ 
\tilde{x} = \pi_{\Omega_{1}}\bfone +
(1-\alpha)^{-1}\pi_{\Omega_{<1}}\hat{x}.
\]

Assume for contradiction that $\tilde{x}$ is infeasible. Then there
must be some constraint $j$ for which $A_{j\cdot} \tilde{x} < b_j$. Using
the decomposition above and the fact that $\alpha \in (0,1)$, we have
\begin{equation} \label{eq:pf1}
A_{j\cdot}\pi_{\Omega_{< 1}}\hat{x} < (b_j - A_{j\cdot}\pi_{\Omega_{1}}
\bfone)(1-\alpha).
\end{equation}
On the other hand, by \eqnok{eq:feas:bound}, we have
\[ 
A_{j\cdot} (\pi_{\Omega_{1}}\bfone + \pi_{\Omega_{< 1}}\hat{x}) \ge
A_{j\cdot} (\pi_{\Omega_{1}}\hat{x} + \pi_{\Omega_{< 1}}\hat{x}) \ge b_j
- \epsilon.
\]
and so
\begin{equation} \label{eq:pf2}
A_{j\cdot} \pi_{\Omega_{< 1}} \hat{x} \geq b_j - A_{j\cdot}
\pi_{\Omega_{1}}\bfone - \epsilon
\end{equation}
By combining \eqnok{eq:pf1} and \eqnok{eq:pf2}, we obtain
\[ 
(b_j - A_{j\cdot}\pi_{\Omega_1}\bfone)(1-\alpha) >  (b_j - A_{j\cdot}\pi_{\Omega_1}\bfone) - \epsilon  
\]
Since $b_{j} - A_{j\cdot} \pi_{\Omega_1} \bfone \ge b_j - A_{j \cdot}
\tilde{x} > 0$, we can divide by $b_{j} - A_{j\cdot} \pi_{\Omega_1}
\bfone$ without changing signs to obtain
\begin{equation} \label{eq:pf3}
\frac{\epsilon}{b_j - A_{j\cdot}\pi_{\Omega_1}\bfone} > \alpha \;\;
\Rightarrow \;\;
b_j - A_{j\cdot}\pi_{\Omega_1}\bfone < \epsilon/\alpha.
\end{equation}
We have by using the definition of $\alpha$ that $b_j -
A_{j\cdot}\pi_{\Omega_1}\bfone \ge q(P)\ge \epsilon/\alpha$, since
$\pi_{\Omega}\bfone$ is an integral but infeasible point for (P). This
fact contradicts \eqnok{eq:pf3}, so we have proved feasibility of
$\tilde{x}$ for (P).

We now bound the difference between $c^* \tilde{x}$ and $c^*x^*$,
where $x^*$ is the optimal solution of (P). Since $\tilde{x}$ is
feasible, we have that $c^{T}x^* \leq c^{T}\tilde{x}$. For the upper
bound, we have
\[ 
c^{T}\tilde{x} - c^{T}x^* \leq (1-\alpha)^{-1}c^{T} \hat{x} - c^{T}x^*
\leq (1-\alpha)^{-1}(c^{T} \hat{x} - c^{T}x^*) \leq
\frac{\delta}{1-\alpha} c^{T}x^*.
\]
The first inequality follows from $c^{T}z \geq c^{T}
(\Pi_{[0,1]^{n}}z)$ since $c \geq 0$; the second inequality is from
$\alpha \in (0,1)$; and the third inequality follows from the fact
that $\hat{x}$ is a $(\epsilon, \delta)$ approximation.
\end{proof}

In our experiments, we set $\alpha= \epsilon/ q(P)$, which is computed
using the approximate $(\epsilon, \delta)$ optimal fractional
solution.

\subsection{Rounding for Packing}

A packing problem is a maximization linear program $P(A,b,c)$ where
$A,b,c\geq 0$ and $A \in \mathbb{Z}^{m \times n}$, $b \in
\mathbb{Z}^{m}$, and $c \in \mathbb{Z}^{n}$ along with bound
constraints $[0,1]$ on all variables. That is,
\[ 
\max_{u \in [0,1]^{m}} u^{T}b \quad \mbox{subject to} \;\;  A^{T}u \leq c. 
\qquad\qquad \mbox{[$P(A,b,c)$]}
\]
In this class of problems, we can assume without loss of generality
that $c \geq \bfone$.  The equality constrained formulation of this
problem is
\[ 
\max_{u \in [0,1]^{m}, z \in \R^{n}} u^{T}b \quad \mbox{subject to} \;\;  A^{T}u + z = c, \; z \geq 0.
\]
(We can set $z = \max(c-A^Tu,0)$ to obtain the equivalence.)

We use $A_{\cdot i}$ to denote the $i$th column of $A$ in the
discussion below.


\begin{lemma}
Let $P$ be a packing program. Let $\hat{u}$ be an
$(\epsilon, \delta)$-approximate solution of $P$, then there is a
$(0,\frac{\delta + \alpha}{1+\alpha})$-approximate solution
$\tilde{u}$ defined as
\[  \tilde{u} = \hat{u}/ (1+\alpha) \]
provided that $\hat{u} \in [0,1]^{m}$ where $\alpha \ge \epsilon / \left(
\min_{i=1,2,\dotsc,n} c_i \right)$.
\end{lemma}
\begin{proof}
We observe first that $\tilde{u} \in [0,1]^m$. To prove that
$A^T\tilde{u} \le c$, note that since $\hat{u}$ is an
$(\epsilon,\delta)$-approximate solution, we have
\[
A^T\hat{u} \le c + \epsilon \bfone
\le c + \alpha \left( \min_{l=1,2,\dotsc,n} c_l \right) \bfone \le (1+\alpha) c,
\]
proving the claim.


Let $u_{*}$ be an optimal solution of $P(A,b,c)$. Since $\tilde{u}$ is
feasible and this is a maximization problem, we have $u_{*}^{T}b \geq
\tilde{u}^{T}b \geq 0$. For the other bound, we have
\[ 
u_{*}^{T}b - \tilde{u}^{T}b 
= u_{*}^{T}b - \frac{1}{1+\alpha} \hat{u}^{T}b
\leq  u_{*}^{T}b - \frac{1-\delta}{1+\alpha} u_{*}^{T}b 
= \frac{\delta + \alpha}{1+\alpha} u_{*}^{T}b,
\]
completing the proof.
\end{proof}

A quick examination of the proof suggests that we can take $\alpha :=
\left(\max_{i=1,2,\dotsc,n} \frac{A_{\cdot i}^T \hat{u} -
c_i}{c_i}\right)_{+}$, which is never larger than $\alpha$ as defined above. In
our experiments, we set $\alpha$ using this tighter bound and $\theta =
\frac{1}{k}$ in algorithm \ref{alg:set_packing_rounding_2}.  We note that the
algorithm is sensitive to the value of $\theta$. Any positive value of $\theta
k \geq 1$ will always return a valid independent set. The proofs in
\cite{Bansal2012} require that $\theta$ must be greater or equal to $1$, but we
found that $\theta = \frac{1}{k}$ works much better in practice.

\subsection{Rounding for Multiway-Cuts}

Since we enforce the simplex constraints in the SCD solve, every
solution obtained by our quadratic relaxation is automatically
feasible for our linear program.

\section{Linear Programming Condition Numbers}
\label{app:lp:conditioning}

In this section, we describe estimates of $(\delta_P,\delta_{D})$ in
detail for vertex cover, and sketch the ideas for estimating these
quantities for the other relaxations that we consider in this paper.

\subsection{Vertex Cover: The Bounds in Detail}

Consider vertex cover with a graph $G=(V,E)$, where $|V| = n$ and
$|E|=m$. The LP relaxation is as follows
\begin{equation} \label{eq:vc.again}
\min_{x \in \R_{+}^{n}} \bfone^Tx \quad \mbox{subject to} \;\; x_v +
x_w \geq 1 \text{ for all } (v,w) \in E \text{ and } x_v \leq 1 \text{
  for all } v \in V.
\end{equation}
The dual of this program is
\[ 
\max_{u \in \R_{+}^{m}, z \in \R^{+}} u^{T}\bfone - z^{T}\bfone \quad
\mbox{subject to} \;\; \sum_{e : e \ni v} u_e - z_{v} \leq 1 \text{
  for each } v \in V.
\]

\paragraph{Computing $\|d\|$.} 
Define $\|d\| = \max \{\|A\|_{F},\|b\|_2,\|c\|_2\} $ for this problem,
where $(A,b,c)$ are the data defining \eqnok{eq:vc.again}. We have
\[ 
\|A\|_{F} = \sqrt{2m + n}, \quad \|b\|_2 = \sqrt{m+n} \quad \|c\| = \sqrt{n}  
\]
Hence, $\|d\| = \sqrt{2m + n}$.


\paragraph*{Primal Bound.}

We define $x = \frac{2}{3} \bfone$, and figure how large a
perturbation $(\Delta A, \Delta b, \Delta c)$ is needed to problem
data $(A,b,c)$ to make this particular point infeasible. The norm of
this quantity will give a lower bound on the distance to
infeasibility.

By construction of $x$, we have that $Ax - b = \frac{1}{3} \bfone$.
For infeasibility with respect to one of the cover constraints, we
would need for some $i$ that 
\[
|(\Delta A)_{i \cdot} x  - \Delta b_i | \ge \frac13,
\]
which, given our definition of $x$, would require
\begin{equation} \label{eq:cc1}
\frac23 \sum_{j=1}^n | \Delta A_{ij} | + | \Delta b_i| > \frac13.
\end{equation}
We must therefore have that 
\[
\sum_{j=1}^n | \Delta A_{ij} | \ge \frac14 \;\; \mbox{and/or}
\;\;
 | \Delta b_i| > \frac16.
\]
In the first case, noting that 
\[ 
\frac{1}{4} n^{-1/2} = \min_{z \in [0,1]^{n}} \norm{z}_2  \quad \mbox{subject to} \;\;
z^{T} \bfone \geq \frac{1}{4},
\]
we would have that $\|\Delta A\|_F \ge \| (\Delta A)_{i \cdot} \|_2
\ge n^{-1/2}/4$. In the second case, we would have $\| \Delta b \|_2
\ge | \Delta b_i| \ge 1/6$. 

Suppose that the infeasibility happens instead with respect to one of
the $x \le \bfone$ constraints. A similar argument for the violated
constraint would lead to the same necessary condition \eqnok{eq:cc1}
and the same bounds.

In either case, assuming that $n \ge 3$, we have
\[
\| (\Delta A, \Delta b, \Delta c) \| \ge n^{-1/2}/4,
\]
so that 
\[ 
\delta_{P} \geq \norm{d}^{-1} n^{-1/2}/4 .
\]

\paragraph*{Dual Bound.} 

We consider here a fixed vector $(u,z) = 0$.  For infeasibility, we
would need $\Delta c_i<-1$ for some $i$, and therefore $\| \Delta d \|
\ge 1$.We thus have
\[ 
\delta_{D} \geq \norm{d}^{-1}   
\]

Putting the primal and dual bounds together, and using our bound on
$\|d\|$, we obtain
\[ 
\frac{1}{\delta_{P}\delta_{D}} = O( \|d\|^2 n^{1/2} ) = O( (m+n) n^{1/2} ).
\]

\subsection{Packing and Covering Programs}

Suppose we have a covering program with data $(A,b,c) \geq 0$, with
$[0,1]$ bound constraints on each variable. That is,
\[ 
\min_{x \in \R_{+}^{n}} c^{T}x \quad \mbox{subject to} \;\;  Ax \geq b, x \leq \bfone,
\]
its dual is a packing program:
\[ 
\max_{u \in \R_{+}^{m}, z \in \R_{+}^{n}} u^{T}b - z^{T}\bfone
\quad \mbox{subject to} \;\;  A^{T}u - z \leq c.
\]

Generalizing our argument above, we find a point that has the most
slack from each constraint. Defining the following measure of slack:
\[
s(A,b,c) = \max_{x \in \mathbb{R}_{+}^{n} : Ax \geq b, x \leq \bfone}
\min \set{\min_{i=1,\dots,n} 1 - x_i, \min_{j=1,\dots,m} b_j -
  A_{j\cdot} x},
\]
we can obtain a lower bound $\delta_{P} \geq \norm{d}^{-1} n^{-1/2}
s(A,b,c)/2$, as follows. Suppose that $x_S$ is the point that achieves
the maximum slack.  We need that one of the following conditions holds
for at least one constraint $i$: $\Delta A_{i \cdot } x_S > s(P)/2$ or
$| \Delta b_i| \ge s(P)/2$. Observe that
\[ 
\Delta A_{i \cdot } x_S \leq \|x_{S}\|_2 \|\Delta A_{i \cdot }\|_2
\leq n^{1/2} \|\Delta A_{i}\|_2.
\]
(The second inequality follows from $0 \leq x_{S} \leq \bfone$.) Thus,
in this case, $\|\Delta A_{i}\|_2 > s(A,b,c) n^{-1/2}/2$.  Using a similar
argument to the previous subsection, we have
\[
\delta_P \ge \|d\|^{-1} \| \Delta d\| \ge \|d\|^{-1} s(A,b,c) n^{-1/2}/2.
\]

Since $(u,z)=(0,0)$ is feasible for the dual, we have by a similar
argument to the previous subsection that infeasibility occurs only if
$| \Delta c_i| \ge c_i$ for at least one $i$. We therefore have
$\| \Delta d \| \ge  \min_{i=1,2,\dotsc,n} c_i$, so that 
\[
\delta_D \ge \|d\|^{-1} \min_{i=1,2,\dotsc,n} c_i.
\]

Putting the bounds on $\delta_P$ and $\delta_D$ together, we have
\[
\frac{1}{\delta_P \delta_D}  \le \|d\|^2 
\frac{1}{s(A,b,c) \min_{i=1,2,\dotsc,n} c_i} O(n^{1/2}).
\]

\section{Extended Experimental Results} \label{app:experiments}

In this section, we elaborate our discussion on the experimental results in
Section~\ref{sec:expt_cplex} and provide additional evidence to support our
claims. Figures~\ref{fig:app:cplex_lp} and \ref{fig:app:cplex_lp_default}
compare the performance of Thetis with \cplex-IP and \cplex-LP on all tested
instances of vertex cover, independent set, and multiway-cut. In all three
formulations, we used unit costs in the objective function. The results in
Figure~\ref{fig:app:cplex_lp_default} were obtained by using default tolerance
on \cplex-LP, while Figure~\ref{fig:app:cplex_lp} uses the same tolerance
setting as the main manuscript. 

\begin{figure}[h]
\centering
\footnotesize
\setlength{\tabcolsep}{3pt}
\begin{tabular}{ | c | c c c| c c c | c c c | }
\hline
{\bf VC} & \multicolumn{3}{c|}{ \cplex IP} & \multicolumn{3}{c|}{ \cplex LP}
& \multicolumn{3}{c|}{ Thetis } \\ \cline{2-10}
(min) & t (secs) & BFS & Gap(\%) & t (secs) & LP & RSol & t (secs) & LP & RSol \\ \hline
frb59-26-1 & - & 1475 & 0.7 & 2.48 & 767.0 & 1534 & 0.88 & 959.7 & 1532\\
frb59-26-2 & - & 1475 & 0.6 & 3.93 & 767.0 & 1534 & 0.86 & 979.7 & 1532\\
frb59-26-3 & - & 1475 & 0.5 & 4.42 & 767.0 & 1534 & 0.89 & 982.9 & 1533\\
frb59-26-4 & - & 1475 & 0.5 & 2.65 & 767.0 & 1534 & 0.89 & 983.6 & 1531\\
frb59-26-5 & - & 1475 & 0.5 & 2.68 & 767.0 & 1534 & 0.90 & 979.4 & 1532\\
Amazon & 85.5 & 1.60\e{5} & - & 24.8 & 1.50\e{5} & 2.04\e{5} & 2.97 & 1.50\e{5} & 1.97\e{5}\\
DBLP & 22.1 & 1.65\e{5} & - & 22.3 & 1.42\e{5} & 2.08\e{5} & 2.70 & 1.42\e{5} & 2.06\e{5}\\
Google+ & - & 1.06\e{5} & 0.01 & 40.1 & 1.00\e{5} & 1.31\e{5} & 4.47 & 1.00\e{5} & 1.27\e{5}\\
\hline
{\bf MC} & \multicolumn{3}{c|}{ \cplex IP} & \multicolumn{3}{c|}{ \cplex LP}
& \multicolumn{3}{c|}{ Thetis} \\ \cline{2-10}
(min) & t (secs) & BFS & Gap(\%) & t (secs) & LP & RSol & t (secs) & LP & RSol \\ \hline
frb59-26-1 & 72.3   & 346 & - & 312.2 & 346 & 346 & 5.86 & 352.3 & 349\\
frb59-26-2 & 561.1 & 254 & - & 302.9 & 254 & 254 & 5.82 & 262.3 & 254\\
frb59-26-3 & 27.7  & 367 & - & 311.6 & 367 & 367 & 5.86 & 387.7 & 367\\
frb59-26-4 & 65.4  & 265 & - & 317.1 & 265 & 265 & 5.80 & 275.7 & 265\\
frb59-26-5 & 553.9 & 377 & - & 319.2 & 377 & 377 & 5.88 & 381.0 & 377\\
Amazon & - & 12 & NA & - & - & - & 55.8 & 7.3 & 5\\
DBLP & - & 15 & NA & - & - & - & 63.8 & 11.7 & 5\\
Google+ & - & 6 & NA & - & - & - & 109.9 & 5.8 & 5\\
\hline
{\bf MIS} & \multicolumn{3}{c|}{ \cplex IP} & \multicolumn{3}{c|}{ \cplex LP} &
\multicolumn{3}{c|}{ Thetis} \\ \cline{2-10}
(max) & t (secs) & BFS & Gap(\%) & t (secs) & LP & RSol & t (secs) & LP & RSol \\ \hline
frb59-26-1 & - & 50 & 18.0 & 4.65 & 767 & 15 & 0.88 & 447.7 & 18\\
frb59-26-2 & - & 50 & 18.0 & 4.74 & 767 & 17 & 0.88 & 448.6 & 17\\
frb59-26-3 & - & 52 & 13.4 & 3.48 & 767 & 19 & 0.87 & 409.2 & 19\\
frb59-26-4 & - & 53 & 11.3 & 4.41 & 767 & 18 & 0.90 & 437.2 & 17\\
frb59-26-5 & - & 51 & 15.6 & 4.43 & 767 & 18 & 0.88 & 437.0 & 18\\
Amazon & 35.4 & 1.75\e{5} & - & 23.0 & 1.85\e{5} & 1.56\e{5} & 3.09 & 1.73\e{5} & 1.43\e{5}\\
DBLP & 17.3 & 1.52\e{5} & - & 23.2 & 1.75\e{5} & 1.41\e{5} & 2.72 & 1.66\e{5} & 1.34\e{5}\\
Google+ & - & 1.06\e{5} & 0.02 & 44.5 & 1.11\e{5} & 9.39\e{4} & 4.37 & 1.00\e{5} & 8.67\e{4}\\ \hline
\end{tabular}
\caption{Wall-clock time and quality of fractional and integral solutions for
three graph analysis problems using Thetis, \cplex-IP and \cplex-LP. Each code
was given a time limit of one hour, with `-' indicating a timeout.  BFS is the
objective value of the best integer feasible solution found by \cplex-IP.  The
gap is defined as (BFS$-$BB)/BFS where BB is the best known solution bound
found by \cplex-IP within the time limit.  A gap of `-' indicates that the
problem was solved to within $0.01\%$ accuracy and NA indicates that \cplex-IP
was unable to find a valid solution bound.  LP is the objective value of the LP
solution, and RSol is objective value of the rounded solution.}
\label{fig:app:cplex_lp}
\end{figure}


\paragraph*{Maximum Independent Set.} 

We observed that the rounded feasible solutions obtained using Thetis were of
comparable quality to those obtained by rounding the more accurate
solutions computed by \cplex-LP.  The integral solutions obtained from
\cplex-IP were only marginally better than that obtained by LP-rounding, but at
a cost of at least an order of magnitude more time. 

\paragraph*{Multiway Cuts.} 

The number of variables in the multiway-cut problem is $O((|E|+|V|) \times k)$
where $|E|$ is the number of edges, $|V|$ is the number of vertices and $k$ is
the number of terminals.  The terminals were chosen randomly to be in the same
connected component of the graph. All the test instances, excepting Google+,
were fully connected.  For Google+, 201949 (of 211186 vertices) were connected
to the terminals.  For all instances, including Google+, all codes were run on
\eqref{model:lp_multiway_cut} built using the entire graph.

We solved the QP-approximation of \eqref{model:lp_multiway_cut} using a
block-SCD method, which is variant of Algorithm \ref{alg:SCD}, in which an
update step modifies a block of co-ordinates of size $k$.  For the blocks
corresponding to variables $x_v,\ \forall v \in V$, we performed a projection
on to the $k$-dimensional simplex $\Delta_k$. The simplex projection was
necessary to ensure that the approximate LP solution is always feasible for
\eqref{model:lp_multiway_cut}.  We disabled presolve for Thetis to prevent the
simplex constraints from being eliminated or altered. We did not disable
presolve for \cplex-LP or \cplex-IP. 

Our results demonstrate that Thetis is much more scalable than both \cplex-IP
and \cplex-LP. Thetis was an order of magnitude faster than \cplex-LP on the
Bhoslib instances while generating solutions of comparable quality. Both Thetis
and \cplex-LP recovered the optimal solution on some of the instances. On the
SNAP instances, both \cplex-IP and \cplex-LP failed to complete within an hour
on any of the instances. \cplex-IP was able to generate feasible solutions
using its heuristics, but was able to unable to solve the root-node relaxation
on any of the SNAP instances.

\begin{figure}[h]
\centering
\footnotesize
\setlength{\tabcolsep}{3pt}
\begin{tabular}{ | c | c c c | c c c | c c c | }
\hline
{\bf VC} & \multicolumn{3}{c|}{ \cplex-IP} & \multicolumn{3}{c|}{ \cplex-LP
(default tolerances)} & \multicolumn{3}{c|}{ Thetis} \\ \cline{2-10}
(min) & t (secs) & BFS & Gap(\%) & t (secs) & LP & RSol & t (secs) & LP & RSol \\ \hline
frb59-26-1 & - & 1475 & 0.7 & 4.59 & 767.0 & 1534 & 0.88 & 959.7 & 1532\\
frb59-26-2 & - & 1475 & 0.6 & 4.67 & 767.0 & 1534 & 0.86 & 979.7 & 1532\\
frb59-26-3 & - & 1475 & 0.5 & 4.76 & 767.0 & 1534 & 0.89 & 982.9 & 1533\\
frb59-26-4 & - & 1475 & 0.5 & 4.90 & 767.0 & 1534 & 0.89 & 983.6 & 1531\\
frb59-26-5 & - & 1475 & 0.5 & 4.72 & 767.0 & 1534 & 0.90 & 979.4 & 1532\\
Amazon & 85.5 & 1.60\e{5} & - & 21.6 & 1.50\e{5} & 1.99\e{5} & 2.97 & 1.50\e{5} & 1.97\e{5}\\
DBLP & 22.1 & 1.65\e{5} & - & 23.7 & 1.42\e{5} & 2.07\e{5} & 2.70 & 1.42\e{5} & 2.06\e{5}\\
Google+ & - & 1.06\e{5} & 0.01 & 60.0 & 1.00\e{5} & 1.30\e{5} & 4.47 & 1.00\e{5} & 1.27\e{5}\\
\hline
{\bf MC} & \multicolumn{3}{c|}{ \cplex-IP} & \multicolumn{3}{c|}{ \cplex-LP (default tolerances) } & \multicolumn{3}{c|}{ Thetis ($\epsilon=0.1$)} \\ \cline{2-10}
(min) & t (secs) & BFS & Gap(\%) & t (secs) & LP & RSol & t (secs) & LP & RSol \\ \hline
frb59-26-1 & 72.3  & 346 & - & 397.9 & 346 & 346 & 5.86 & 352.3 & 349\\
frb59-26-2 & 561.1 & 254 & - & 348.1 & 254 & 254 & 5.82 & 262.3 & 254\\
frb59-26-3 & 27.7  & 367 & - & 386.6 & 367 & 367 & 5.86 & 387.7 & 367\\
frb59-26-4 & 65.4  & 265 & - & 418.9 & 265 & 265 & 5.80 & 275.7 & 265\\
frb59-26-5 & 553.9 & 377 & - & 409.6 & 377 & 377 & 5.88 & 381.0 & 377\\
Amazon & - & 12 & NA & - & - & - & 55.8 & 7.28 & 5\\
DBLP & - & 15 & NA & - & - & - & 63.8 & 11.70 & 5\\
Google+ & - & 6 & NA & - & - & - & 109.9 & 5.84 & 5\\
\hline
{\bf MIS } & \multicolumn{3}{c|}{ \cplex-IP} & \multicolumn{3}{c|}{ \cplex-LP
(default tolerances)} & \multicolumn{3}{c|}{ Thetis ($\epsilon=0.1$)} \\
\cline{2-10}
(max) & t (secs) & BFS & Gap(\%) & t (secs) & LP & RSol & t (secs) & LP & RSol \\ \hline
frb59-26-1 & - & 50 & 18.0 & 4.88 & 767 & 16 & 0.88 & 447.7 & 18\\
frb59-26-2 & - & 50 & 18.0 & 4.82 & 767 & 16 & 0.88 & 448.6 & 17\\
frb59-26-3 & - & 52 & 13.4 & 4.85 & 767 & 16 & 0.87 & 409.2 & 19\\
frb59-26-4 & - & 53 & 11.3 & 4.67 & 767 & 15 & 0.90 & 437.2 & 17\\
frb59-26-5 & - & 51 & 16.6 & 4.82 & 767 & 16 & 0.88 & 437.0 & 18\\
Amazon & 35.4 & 1.75\e{5} & - & 25.7 & 1.85\e{5} & 1.58\e{5} & 3.09 & 1.73\e{5} & 1.43\e{5}\\
DBLP & 17.3 & 1.52\e{5} & - & 24.0 & 1.75\e{5} & 1.41\e{5} & 2.72 & 1.66\e{5} & 1.34\e{5}\\
Google+ & - & 1.06\e{5} & 0.02 & 68.8 & 1.11\e{5} & 9.40\e{4} & 4.37 & 1.00\e{5} & 8.67\e{4}\\
\hline
\end{tabular}
\caption{Wall-clock time and quality of fractional and integral solutions for
three graph analysis problems using Thetis, \cplex-IP and \cplex-LP (run to
default tolerance). Each code was given a time limit of one hour, with `-'
indicating a timeout.  BFS is the objective value of the best integer feasible
solution found by \cplex-IP.  The gap is defined as (BFS$-$BB)/BFS where BB is
the best known solution bound found by \cplex-IP within the time limit.  A
gap of `-' indicates that the problem was solved to within $0.01\%$ accuracy
and NA indicates that \cplex-IP was unable to find a valid solution bound.  LP
is the objective value of the LP solution, and RSol is objective value of the
rounded solution.}
\label{fig:app:cplex_lp_default}
\end{figure}

\end{document}